\documentclass[11pt, twoside, leqno]{amsart}  
\usepackage{soul}
\usepackage{lipsum}
\usepackage{amsfonts}
\usepackage{graphicx}
\usepackage{epstopdf}
\usepackage{calligra}
\usepackage{amsfonts,amsmath,amsthm,amssymb}
\usepackage{mathtools}
\usepackage{hyperref}
\usepackage{autonum}
\usepackage{hhline}
\usepackage{array}
\usepackage{diagbox}
\usepackage{tcolorbox}
\usepackage{mdframed}
\usepackage{multicol}
\usepackage{graphicx}
\usepackage{subcaption}
\usepackage{moreverb}
\usepackage{bbm}
\usepackage[margin=1.34in]{geometry}
\usepackage{todonotes}
\usepackage{scalerel,amssymb}
\allowdisplaybreaks
\usepackage{mathrsfs}  
\usepackage{lineno}
\usepackage{todonotes}
\usepackage{tikz}
\usepackage{pgfplots}
\usetikzlibrary{arrows.meta}
\usepackage[numbers,sort&compress]{natbib}
\definecolor{mygreen}{HTML}{43a047}
\usepackage{subcaption}
\usepackage{doi}
\usepackage{alphalph}
\usepackage{booktabs}
\usepackage{makecell}
\usepackage{adjustbox}
\usepackage{appendix}
\usepackage{amsaddr}
\usepackage{enumitem}
\usepackage{array}
\newcolumntype{H}{>{\setbox0=\hbox\bgroup}c<{\egroup}@{}}

\def\calUW{\mathcal{U}^{\textup{W}}}
\def\calUBK{\mathcal{U}^{\textup{KB}}}
\def\calN{\mathcal{N}}
\def\calB{\mathcal{B}}
\def\calBW{\mathcal{B}^{\textup{W}}}
\def\calT{\mathcal{T}}
\def\calF{\mathcal{F}}

\def\ulaaa{\underline{\mathfrak{a}}}
\def\olaaa{\overline{\mathfrak{a}}}

\def\Clin{C_{\textup{lin}}}
\newcommand{\Om}{\Omega}

\newcommand{\Dt}{\textup{D}_t}

\def\aaa{\mathfrak{a}}
\def\bbb{\mathfrak{b}}

\def \xin{\xi^{(n)}}
\def \xit{\xi^{(n)}_t}
\def \xitt{\xi^{(n)}_{tt}}
\def \xittt{\xi^{(n)}_{ttt}}
\def \bxin{\boldsymbol{\xi}}
\def \bxit{\boldsymbol{\xi_t}}
\def \bxitt{\boldsymbol{\xi_{tt}}}
\def \bxittt{\boldsymbol{\xi_{ttt}}}
\def \ut{u_t}
\def \utt{u_{tt}}
\def \uttt{u_{ttt}}
\def \psit{\psi_t}
\def \psitt{\psi_{tt}}

\def \un{u^{(n)}}
\def \unt{u^{(n)}_t}
\def \untt{u^{(n)}_{tt}}
\def \unttt{u^{(n)}_{ttt}}
\def \taua {\tau^a}

\newcommand{\Dal}{{\textup{D}}_t^\alpha}
\newcommand{\Doal}{{\textup{D}}_t^{1-\alpha}}
\newcommand{\bmu}{\boldsymbol{\mu}}
\newcommand{\bxi}{\boldsymbol{\xi}}
\newcommand{\dt}{\, \textup{d} t}
\newcommand{\ds}{\, \textup{d} s }
\newcommand{\dx}{\, \textup{d} x}

\newcommand{\dxs}{\, \textup{d}x\textup{d}s}
\newcommand{\dxt}{\, \textup{d}x\textup{d}t}
\newcommand{\intTO}{\int_0^T \int_{\Omega}}
\newcommand{\inttO}{\int_0^t \int_{\Omega}}
\newcommand{\intt}{\int_0^t}

\newcommand{\intO}{\int_{\Omega}}

\newcommand{\R}{\mathbb{R}} 
\newcommand{\N}{\mathbb{N}} 
\newcommand{\Ltwo}{L^2(\Omega)}

\newcommand{\Hnn}{H^{-1}}
\newcommand{\Hneg}{H^{-1}(\Omega)}
\newcommand{\Honezero}{H_0^1(\Omega)}
\newcommand{\Hthree}{H^3(\Omega)}
\newcommand{\Honetwo}{{H^2(\Omega)\cap H_0^1(\Omega)}}
\newcommand{\Honethree}{H^3_\diamondsuit(\Omega)}

\newtheorem{theorem}{Theorem}
\newtheorem{lemma}{Lemma}
\newtheorem{proposition}{Proposition}

\newtheorem*{assumption*}{Assumptions}

\newtheorem{remark}{Remark}
\numberwithin{lemma}{section}
\numberwithin{proposition}{section}
\numberwithin{theorem}{section}
\numberwithin{equation}{section}
\newcommand{\bfq}{\boldsymbol{q}}
\newcommand{\bfmu}{\boldsymbol{\mu}}
\newcommand{\bfxi}{\boldsymbol{\xi}}

\newcommand{\frakK}{\mathfrak{K}}
\newcommand{\tfrakK}{\tilde{\mathfrak{K}}}
\newcommand{\frakKone}{\mathfrak{K}}
\newcommand{\frakKtwo}{\mathfrak{K}_2}
\newcommand{\tfrakKone}{\tilde{\mathfrak{K}}_1}
\newcommand\Lconv{\ast}

\newcommand{\rt}{\tau_\theta}
\newcommand{\ttb}{\tau_\theta^b}

\definecolor{grey}{rgb}{0.5,0.5,0.5}
 
\usepackage{stackengine,scalerel}

\definecolor{darkgreen}{rgb}{0,0.5,0}

\makeatletter
\@namedef{subjclassname@2020}{\textup{2020} Mathematics Subject Classification}
\makeatother
\makeatletter
\newcommand{\leqnomode}{\tagsleft@true}
\newcommand{\reqnomode}{\tagsleft@false}
\makeatother  
\newcommand{\revise}[1]{\textcolor{black}{#1}}
\title[Nonlinear acoustic equations of fractional higher order]{Nonlinear acoustic equations of fractional higher order at the singular limit}
 \author[V.\ Nikoli\'{c}]{\vspace*{-4mm}\small Vanja Nikoli\'{c}}
\address{ 
	Department of Mathematics, 
	Radboud University   \\ 
	Heyendaalseweg 135,
	6525 AJ Nijmegen, The Netherlands}
\email{vanja.nikolic@ru.nl}
\begin{document}
\vspace*{8mm}
\begin{abstract}
When high-frequency sound waves travel through media with anomalous diffusion, such as biological tissues, their motion can be described by nonlinear wave equations of fractional higher order. These can be understood as nonlocal generalizations of the Jordan--Moore--Gibson--Thompson equations in nonlinear acoustics. In this work, we relate them to the classical second-order acoustic equations and, in this sense, justify them as their approximations for small relaxation times. To this end, we perform the singular limit analysis for a class of corresponding nonlocal wave models and determine their behavior as the relaxation time tends to zero. We show that, depending on the nonlinearities and assumptions on the data, these models can be seen as approximations of the Westervelt, Blackstock, or Kuznetsov wave equations. The analysis rests upon the uniform bounds for the solutions of the equations with fractional higher-order derivatives, obtained through a testing procedure tailored to the coercivity property of the involved (weakly) singular memory kernels.
\end{abstract}
\vspace*{-8mm}
\maketitle           
\section{Introduction} \label{Sec:Introduction}
Ultrasound propagation through complex tissue-like media is known to follow more involved governing laws than in fluids~\cite{holm2019waves, prieur2011nonlinear, prieur2012more}. This evolution is nonlocal in nature, as the past may influence the present state.  At high frequencies or intensities, nonlinear effects come additionally into play. These modeling considerations are especially relevant in medical applications of ultrasonic waves in imaging~\cite{szabo2004diagnostic} and cancer therapy~\cite{kennedy2005high}. \\
\indent Motivated by the accurate description of nonlinear and nonlocal acoustic phenomena, we investigate a class of equations of the following type: 
\begin{equation}  \label{general_wave_equation}
	\begin{aligned}
\begin{multlined}[t]
	\taua \frakKone *\uttt+\aaa \utt
-c^2 \bbb \Delta u -\delta \Delta \ut -\taua c^2 \frakK*\Delta \ut   +\calN=f,
\end{multlined}
	\end{aligned}
\end{equation}
where $*$ denotes the Laplace convolution operator in time. These involve a kernel $\frakK$ of fractional type and general nonlinearities in the form of
\[
\aaa=\aaa(u, u_t), \quad \bbb=\bbb(u_t), \quad \calN=\calN(\ut, \nabla u, \nabla \ut).
\]
They were derived in \cite{kaltenbacher2022time}, under the name fractional Jordan--Moore--Gibson--Thompson (fJMGT) equations, with $\frakK$ being the Abel kernel:
\begin{equation} \label{Abel_kernel}
\frakK(t)= \frac{1}{\Gamma(1-\alpha)}t^{-\alpha}, \quad \alpha \in (1/2, 1) 
\end{equation}
 and the nonlocal terms the Caputo--Djrbashian fractional derivatives; here we will impose a set of assumptions on the kernel that generalizes \eqref{Abel_kernel}. To arrive at \eqref{general_wave_equation}, the Fourier heat flux law is replaced by a nonlocal Maxwell--Cattaneo law within the system of constitutive equations of sound propagation. The latter involves the relaxation time $\tau>0$. This change is responsible for the higher order in the principal term of the resulting acoustic equation; we will come back to this discussion in Section~\ref{Sec:Modeling} with further modeling details. The power $a>0$ in \eqref{general_wave_equation} is kernel-dependent and there to ensure the dimensional homogeneity. For kernels \eqref{Abel_kernel} corresponding to fractional derivatives of Caputo--Djrbashian type, it is equal to the fractional order of differentiation $\alpha$. \\
\indent Naturally, the question arises to which extent these equations can represent reality. As the relaxation time $\tau$ is relatively small, one might expect a certain continuity between the solutions of \eqref{general_wave_equation} and those of the limiting models as $\tau \searrow 0$. Formally setting $\tau$ to zero in \eqref{general_wave_equation} yields the classical strongly damped Kuznetsov,~\cite{kuznetsov1971equations} Blackstock,~\cite{blackstock1963approximate} or Westervelt~\cite{westervelt1963parametric} wave equations in nonlinear acoustics:
\begin{equation} \label{limiting_eq}
\begin{aligned}
\begin{multlined}[t]
\aaa(u, u_t) \utt
-c^2 \bbb(u_t) \Delta u
-  \delta  \Delta \ut +\calN(\ut, \nabla u, \nabla \ut)=f,
\end{multlined}
\end{aligned}
\end{equation}
depending on the choice of the nonlinearities (that is, functions $\aaa$, $\bbb$, and $\calN$). These local-in-time quasilinear wave models have received a lot of attention in the analysis over the recent years and are as a consequence by now mathematically well-understood; we refer to~\cite{fritz2018well, kaltenbacher2009global, meyer2011optimal, tani2017mathematical} for a selection of relevant results on their (local and global) well-posedness analysis. An overview of further related mathematical research in nonlinear acoustics can be found in the review paper~\cite{kaltenbacher2015mathematics}. \\
\indent The goal of this work is to relate the nonlocal and local concepts of describing the nonlinear sound waves by establishing the behavior of solutions to \eqref{general_wave_equation} as the relaxation time tends to zero. Interestingly, equation \eqref{general_wave_equation} should be considered with three initial conditions, whereas the limiting problem is supplemented by two. Thus the nature of the problem changes in the vanishing limit of the relaxation time. As we will see, solutions to \eqref{general_wave_equation} can indeed be seen as approximations of solutions to \eqref{limiting_eq} as $\tau \searrow 0$, provided that the kernel and data satisfy suitable assumptions.\\
\indent To unlock the singular limit analysis, we must first obtain $\tau$-uniform bounds for the solutions of \eqref{general_wave_equation}. The difficulty in deriving these lies in the interplay between the nonlocal and nonlinear aspects of the problem. When developing an energy method for \eqref{general_eq}, the available test functions are restricted by the coercivity one can expect from the memory kernel. At the same time, they should still work well enough to tackle the involved nonlinearities. For example, to ensure that the limiting equation \eqref{limiting_eq} does not degenerate (meaning that it is still a wave equation), we need to guarantee that the function $\aaa=\aaa(u, \ut)$ within the leading term stays uniformly positive. This issue translates to needing $L^\infty(\Omega)$ bounds on the solution $u$ or its time derivative $\ut$, which is in the analysis of nonlinear acoustic equations commonly resolved by having sufficiently smooth and small solutions and using an embedding, such as $H^2(\Omega)\hookrightarrow L^\infty(\Omega)$; see, for example,~\cite{kaltenbacher2009global, kaltenbacher2011wellposedness}.  Obtaining these $H^2(\Omega)$ bounds on $u$ or $\ut$, uniformly in $\tau$, puts an additional strain on already delicate energy arguments inherently needed for fractional-type wave equations.  \\
\indent The key idea of the present work is to see nonlocal equation \eqref{general_eq} in the following form:
\begin{equation} 
\begin{aligned}
\begin{multlined}[t]
\taua \frakKone *(\utt-c^2 \Delta u)_t+\aaa \utt
-c^2 \bbb \Delta u -\delta \Delta \ut  +\calN=f.
\end{multlined}
\end{aligned}
\end{equation}
Grouping the nonlocal terms like this suggests to use \[y(t)=\utt-c^2 \Delta u\] as a test function for the uniform analysis in $\tau$. Such an analysis is tailored to the coercivity property one can expect from fractional-type kernels:
 \begin{equation}
	\int_0^{T} \left(\frakKone* y_t \right)(t) y(t)\dt \geq - C_{\frakKone} |y(0)|^2
\end{equation}
and forms the core of our arguments. This testing is first employed on a suitable linearization of \eqref{general_wave_equation} and then combined with a fixed-point strategy, under the assumption of sufficiently small initial data. To this end, we will distinguish between two types of nonlinearities that we term Westervelt- and Kuznetsov--Blackstock-type here, as they will require different smoothness assumptions on the initial data. \\
\indent In the case that $\frakKone$ is the Dirac delta distribution $\delta_0$, \eqref{general_wave_equation} with $a=1$ reduces to the Jordan--Moore--Gibson--Thompson (JMGT) equation~\cite{jordan2014second, moore1960propagation} with the leading term of third order:
\begin{equation}  \label{JMGT}
\begin{aligned}
\begin{multlined}[t]
\tau (\utt-c^2 \Delta u)_t+\aaa \utt
-c^2 \bbb \Delta u -\delta \Delta \ut    +\calN=f.
\end{multlined}
\end{aligned}
\end{equation}  
\revise{This equation and its linearization, referred to as the Moore--Gibson--Thompson (MGT) equation, have also received plenty of attention in the recent mathematical literature; we refer the reader to~\cite{bongarti2021boundary, bucci2019regularity, chen2019nonexistence, dell2017moore, kaltenbacher2011wellposedness, pellicer2019wellposedness, racke2020global, said2022global} for some of the relevant results}. \revise{Significant progress has also been made in the investigations of global solvability and asymptotic behavior of the (J)MGT equations with additional memory terms; see~\cite{alves2018moore, dell2016moore, kaltenbacher2021inviscid, lasiecka2017global, lasiecka2015moore, lasiecka2016moore, liu2020new} and the references contained therein}. In close relation to the present work, we point out the singular limit analysis of \eqref{JMGT} for the vanishing relaxation time in~\cite{bongarti2020vanishing, kaltenbacher2019jordan, kaltenbacher2020vanishing}, in somewhat simplified settings compared to ours concerning the involved nonlinearities. \revise{Our analysis below also covers the case $\frakK=\delta_0$; as a consequence, we shed new light on the limiting behavior of solutions to this nonlinear third-order equation, in particular, in the presence of the Kuznetsov--Blackstock nonlinearities. The strong convergence analysis for the local JMGT equation with Westervelt nonlinearities can be found in~\cite{bongarti2020vanishing}.} We also point out the concurrent works in~\cite{meliani2023} and~\cite{kaltenbacher2023vanishing} which consider multi-term nonlocal acoustic equations of higher order with two memory kernels, under stronger assumptions on the leading kernel than here. More precisely,~\cite{meliani2023} considers linear equations; \cite{kaltenbacher2023vanishing} allows only for Westervelt-type nonlinearities under stronger assumptions on data than here and without establishing convergence rates in the zero $\tau$ limit.  \\
\indent The rest of the exposition is organized in the following manner. In Section~\ref{Sec:Modeling}, we discuss in more detail the nonlinear acoustic modeling that leads to the equations of fractional higher order studied in this work. In Section~\ref{Sec:Preliminaries}, we specify the assumptions on the kernel and give concrete examples. We then organize the analysis based on the type of nonlinearities in the equation. Section~\ref{Sec:Est_WestType} features the uniform well-posedness analysis in $\tau$ with Westervelt-type nonlinearities, while Section~\ref{Sec:LimWest} establishes their weak and strong limiting behavior. The main results of these sections are Theorems~\ref{Thm:WellP_West} and~\ref{Thm:StrongLim_West}. We then move on to equations with the Kuznetsov--Blackstock nonlinearities. Their uniform analysis in $\tau$ is contained in Section~\ref{Sec:EstBK}, while we investigate their limiting behavior in Section~\ref{Sec:Lim_BK}. The main results of this part are given in Theorems~\ref{Thm:WellP_GFE_BlckstockKuznetov} and~\ref{Thm:StrongLim_Blackstock}. 
\section{Models of ultrasound waves in complex media} \label{Sec:Modeling}
Classical second-order models of nonlinear sound propagation through thermoviscous fluids are based on employing the Fourier flux law within the system of governing equations:
\begin{equation} \label{Fourier_law}
\begin{aligned}
\bfq(t)= -  \kappa \nabla \theta;
\end{aligned}
\end{equation}
see, for example,~\cite{blackstock1963approximate, crighton1979model}. In \eqref{Fourier_law}, $\bfq$ is the heat flux, $\theta$ the absolute temperature, and $\kappa$  the thermal conductivity. A popular alternative to \eqref{Fourier_law} is the Maxwell--Cattaneo law~\cite{cattaneo1958forme}:
\begin{equation} \label{MC_law}
\begin{aligned}
\bfq+\tau \bfq_t= - \kappa \nabla \theta,
\end{aligned}
\end{equation}
which introduces a lag between $\bfq$ and $- \kappa \nabla \theta$ via the thermal relaxation time $\tau>0$, thereby avoiding the infinite speed of propagation. In an attempt to better characterize heat transfer in complex non-homogeneous materials, fractional interpolations of these two laws have been introduced in the literature. In particular,~\cite{compte1997generalized} discusses (among others) the following time-fractional version of the Maxwell--Cattaneo law:
\begin{equation}\label{CM_law}
(1+\tau^\alpha \Dal)\boldsymbol{q} =-\kappa \nabla \theta; \hphantom{{\rt^{1-\alpha}}\Doal }
\end{equation}
see also~\cite{zhang2014time} for a further numerical study involving \eqref{CM_law}. In~\cite{compte1997generalized}, $\Dal$ is understood as the Riemann--Liouville derivative of order $\alpha$, however in the present work we consider it to be 
the Caputo--Djrbashian fractional derivative. We may do so at this point as it is assumed that $\bfq(0)=0$. Given $w \in W^{1,1}(0,T)$, the Caputo--Djrbashian derivative is defined as
\[
\Dt^{\alpha}w(t)=\frac{1}{\Gamma(1-\alpha)}\int_0^t (t-s)^{-\alpha} w_t(s) \, \textup{d} s, \quad \alpha \in(0,1),
\]
where $\Gamma(\cdot)$ denotes the Gamma function; see,~\cite[Ch.\ 1]{kubica2020time} and~\cite[Ch.\ 2.4.1]{podlubny1998fractional}. Having these (fractional) flux laws in mind, we consider below acoustic equations based on the unified law:
\begin{equation} \label{MC_flux_law_nonlocal}
	\bfq+ \taua \frakK*\bfq_t= -  \kappa \nabla \theta, 
\end{equation}
where $*$ denotes the Laplace convolution operator
\[
(\frakK* y)(t):=  \int_0^t	\frakK(t-s)y(s)\ds,
\]
with the kernel $\frakKone$ assumed to be independent of $\tau$. The power $a>0$ is kernel-dependent (but fixed) and there to ensure dimensional homogeneity of the flux law. We will impose conditions on $\frakKone$ in Section~\ref{Sec:Preliminaries} that will allow us to cover both \eqref{MC_law} and \eqref{CM_law}, and in the limit $\tau \searrow 0$ also \eqref{Fourier_law}. \\
\indent In particular, \revise{\eqref{MC_law} follows by setting $\frakK=\delta_0$ and $a=1$ in \eqref{MC_flux_law_nonlocal}}. Time-fractional law \eqref{CM_law} follows by choosing \eqref{Abel_kernel} and setting $a=\alpha$.\\ 
\indent The derivation of nonlinear acoustic equations based on the fractional law in \eqref{MC_flux_law_nonlocal} can be found in~\cite[Sec.\ 2]{kaltenbacher2022time}. The resulting equation was named the fractional Jordan--Moore--Gibson--Thompson (fJMGT) equation with sub-types depending on the involved nonlinearities; see~\cite[Eqs.\ (2.6) and (2.7)]{kaltenbacher2022time}. By retracing the steps of that derivation only now with the generalized heat flux law in \eqref{MC_flux_law_nonlocal} instead of \eqref{CM_law}, the following nonlinear wave equation for the acoustic velocity potential $\psi$ is obtained:
\begin{equation}\label{general_eq}
	\begin{aligned}
\begin{multlined}[t]\taua \frakKone*\psi_{ttt}+\aaa(\psi_t)\psi_{tt}
	-c^2\bbb(\psi_t) \Delta \psi -\delta \Delta \psit-\taua c^2 \Delta \frakKone*\psi_t \\ \hspace*{7cm}+ \calN(\nabla \psi, \nabla \psi_t)=0.
	\end{multlined}
	\end{aligned}
\end{equation}
Above, \[\calN(\nabla \psi, \nabla \psi_t)=\tilde{\ell} \partial_t(|\nabla \psi|)^2=2 \tilde{\ell} \nabla \psi \cdot \nabla \psi_t\]
and either
\begin{equation} \label{Blackstock_nonlinearity}
	\aaa = 1, \quad \bbb(\psi_t)=1-2\tilde{k}\psi_t, 
\end{equation}
or \vspace*{-1mm}
\begin{equation} \label{Kuznetsov_nonlinearity}
	\aaa(\psi_t) = 1+2\tilde{k}\psi_t, \quad \bbb=1. 
\end{equation}
Here $c>0$ denotes the speed of sound and the medium parameter $\delta>0$ is referred to as the sound diffusivity. The nonlinearity coefficients $\tilde{k}$ and $\tilde{\ell}$ are medium-dependent. Equation \eqref{general_eq} can be understood as a generalization of~\cite[Eqs.\ (2.6) and (2.7)]{kaltenbacher2022time}, where $\frakK$ here replaces the fractional-derivative kernel. \\
\indent Formally setting $\tau=0$ with nonlinearities \eqref{Blackstock_nonlinearity} yields the damped Blackstock equation~\cite{blackstock1963approximate} in nonlinear acoustics,
and with \eqref{Kuznetsov_nonlinearity} the Kuznetsov equation.~\cite{kuznetsov1971equations} For the Kuznetsov equation, it is common to employ the approximation \[|\nabla \psi|^2 \approx c^{-2}\psi_t^2,\]
when cumulative nonlinear effects dominate the local ones, and in this manner simplify it by the Westervelt equation~\cite{westervelt1963parametric}; see~\cite[Ch.\ 3]{hamilton1998nonlinear} for a discussion on when local effects can be ignored. Using this approximation in \eqref{general_eq} with \eqref{Kuznetsov_nonlinearity} results in
\begin{equation} \label{West_type_potential}
	\begin{aligned}
		\begin{multlined}[t]
\taua \frakKone*\psi_{ttt}+(1+2k \psi_t)\psi_{tt}
-c^2 \Delta \psi -\delta \Delta \psit-\taua c^2 \frakKone*\Delta \psit=0
		\end{multlined}
	\end{aligned}
\end{equation}
with $k=k+c^{-2}\tilde{\ell} $. It is also common to express the Westervelt equation in terms of the acoustic pressure $p$. Assuming $\frakK \in L^1(0,T)$, formally taking the time derivative of \eqref{West_type_potential} and employing the relation $p=\varrho \psi_t$, where $\varrho$ is the mass density, leads to the pressure form
\begin{equation} \label{JMGT_West_nonlocal_pressure}
	\begin{aligned}
		\begin{multlined}[t]
			\taua \frakKone*(p_{tt}-c^2 \Delta p)_t+(1+2 \underline{k} p)p_{tt}
			- c^2\Delta p	- \delta  \Delta p_{t}+ 2 \underline{k} p_t^2= \mathfrak{r}
		\end{multlined}
	\end{aligned}
\end{equation}
with $\underline{k} =k/\varrho$ and the right-hand side
\begin{equation} \label{source_Westervelt_type}
\mathfrak{r}(t) = -\taua \frakKone(t)(p_{tt}(0)- c^2\Delta p(0)).
\end{equation}
\noindent {\emph{\bf Acoustic models under consideration.}} We will tackle the above different forms of acoustic equations in the analysis by unifying them in one abstract model:
\begin{equation}\label{general_eq_u}
	\begin{aligned}
		\taua \frakKone*\uttt+\aaa \utt
	- \delta \Delta \ut	-c^2\bbb \Delta u -\taua c^2 \frakKone* \Delta \ut
		+ \calN =f,
	\end{aligned}
\end{equation}
and focus on the two distinct nonlinearity types that require different regularity assumptions on the data: \vspace*{1mm}
\begin{itemize}
	\item  Westervelt-type with 
	\[\aaa=\aaa(u)=1+2k_1 u, \quad \bbb=1, \quad \calN=\calN(\ut) = 2k_1 \ut^2;\]
\item  Kuznetsov--Blackstock-type with
	\begin{equation}
	\begin{aligned}
	&\aaa=\aaa(\ut)=1+2k_1 u_t, \quad \bbb=\bbb(\ut)=1-2k_2 \ut, \\[1mm]
	 &\calN=\calN(\nabla u, \nabla \ut) = 2k_3 \nabla u \cdot \nabla \ut,
	\end{aligned}
\end{equation}
\end{itemize}
where we assume $k_{1,2, 3}$ to be real numbers. The Westervelt-type equation incorporates the nonlinearities that arise in \eqref{JMGT_West_nonlocal_pressure}, in which case $u$ denotes the acoustic pressure. The Kuznetsov--Blackstock equation covers \eqref{Blackstock_nonlinearity}, \eqref{Kuznetsov_nonlinearity}, and \eqref{West_type_potential}, where $u$ is the acoustic velocity potential.\\
\indent In all cases, the well-posedness analysis needs to ensure that the leading term $\aaa$ in the limiting equations does not degenerate. As already mentioned, this translates to needing an $L^\infty(\Omega)$ bound; in the case of Westervelt-type equations on $u$, and in the case of Kuznetsov--Blackstock nonlinearities on $\ut$. In the latter case, we also need to be able to control the quadratic gradient term; these two tasks combined lead to needing higher regularity of the solution and in turn higher-order energy arguments compared to the Westervelt case. 

\section{Preliminaries and assumptions on the memory kernel} \label{Sec:Preliminaries}
Throughout this work we assume $\Omega$ to be a smooth bounded domain in $\R^d$, where $d \in \{1,2,3\}$. For the results in Sections~\ref{Sec:Est_WestType} and~\ref{Sec:LimWest} (Westervelt-type nonlinearities) to hold, it is sufficient that $\Omega$ is either a $C^{1,1}$-regular or Lipschitz-regular and convex domain. In Sections~\ref{Sec:EstBK} and~\ref{Sec:Lim_BK} (Kuznetsov--Blackstock-type nonlinearities), $\Omega$ should be a $C^{2,1}$ regular domain. $T>0$ denotes the final propagation time which is given and fixed. \\[2mm]
\noindent \emph{Notation.} 	Below we often use $x\lesssim y$ to denote $x\leq C\, y$ with a constant $C>0$ that does not depend on the thermal relaxation time $\tau$.  We use $\lesssim_T$ to emphasize that the hidden constant $C=C(T)$ tends to $\infty$ as $T \rightarrow \infty$ (often after applying Gronwall's inequality or a Sobolev embedding in time). \\
\indent We frequently omit the spatial and temporal domain when writing norms; for example, $\|\cdot\|_{L^p(L^q)}$ denotes the norm on the Bochner space $L^p(0,T; L^q(\Omega))$. \\[2mm]
\noindent \emph{An auxiliary theoretical result.} Before proceeding, we recall a compactness result from~\cite{meliani2023} which will be helpful in the well-posedness proofs of linearized problems based on the Faedo--Galerkin procedure.
\begin{lemma}[See~\cite{meliani2023}] \label{Lemma:Caputo_seq_compact}
	Let the kernel $\frakK \in L^1(0,T)$ be such that there exists $\tfrakK \in L^{2}(0,T)$ for which $\tfrakK \Lconv \frakK=1$. Consider the space
	\begin{equation} \label{def_Xfp}
	X_\frakK^2(0,T) = \{y \in L^2(0,T) \ |\ \frakK \Lconv y' \in L^2(0,T)\},
	\end{equation}
	with the norm
	\[ \|\cdot\|_{X_\frakK^2} = \big(\|y\|_{L^2}^2 + \|\frakK\Lconv y'\|_{L^2}^2 \big)^{1/2}.\] 
	The following statements hold true:
	\begin{itemize}
		\item 	The space $X_\frakK^2(0,T)$ is reflexive and separable;
		\item The unit ball \revise{$B_{X_\frakK^2}$} in $X_\frakK^2(0,T)$  is weakly sequentially compact;
		\item The space $X_\frakK^2$ continuously embeds into $C[0,T]$.
	\end{itemize}
\end{lemma}
\noindent \emph{Assumptions on the kernel.} Going forward, we make the following assumptions on the memory kernel.\vspace*{1mm}
\begin{enumerate}[label=(\textbf{$\bf \mathcal{A}_\arabic*$})]
	\item \label{Aone}	 	 $\frakK \in \{\delta_0\} \cup L^1(0,T)$;  \\
	\item \label{Atwo} There exists $\tfrakK \in L^2(0,T)$, such that $\frakK*\tfrakK=1$; \\[1mm]
	\item \label{Athree}  There exists a constant $C_{\frakKone}>0$, independent of $\tau$, such that the following coercivity bound holds: 
	\begin{equation}
	\int_0^{T}  \left(\frakKone* y' \right)(t) \,y(t)\dt \geq - C_{\frakKone} |y(0)|^2\quad  \text{for all} \ y\in X_{\frakKone}^2(0, T), 
	\end{equation} 
	where the space $X_{\frakKone}^2(0,T)$ is defined in \eqref{def_Xfp}.
\end{enumerate}
The Dirac delta distribution $\delta_0$ (which satisfies all three assumptions) is included  so that our analysis covers the integer-order Jordan--Moore--Gibson--Thompson equation as well, although we focus on the nonlocal case $\frakKone \in L^1(0,T)$ in the presentation.\\ 
\indent Regularity assumption \ref{Aone} and coercivity assumption \ref{Athree} are satisfied by the fractional kernel
\begin{equation}\label{frac_kernel}
\frakK(t)= \frac{1}{\Gamma(1-\alpha)}t^{-\alpha}, \quad \alpha \in (0,1).
\end{equation}
The latter follows by~\cite[Lemma B.1]{kaltenbacher2021determining} and a density argument (as it is stated in~\cite{kaltenbacher2021determining} for $y \in W^{1,1}(0,t')$).  For this kernel, assumption \ref{Atwo} on the resolvent $\tfrakK$ being in $L^2(0,T)$ is equivalent to assuming that $\alpha>1/2$; see, for example,~\cite[Ch.\ 6]{jin2021fractional}. Therefore our analysis below covers the fractional Jordan--Moore--Gibson--Thompson equation  introduced in~\cite{kaltenbacher2022time}:
\begin{equation} \label{general_wave_equation_Dt}
\begin{aligned}
\begin{multlined}[t]
\tau^{\alpha} \Dal (\utt-c^2 \Delta u)+\aaa \utt
-c^2 \bbb \Delta u -\delta \Delta \ut    +\calN=f
\end{multlined}
\end{aligned}
\end{equation}
with the fractional order of differentiation $\alpha>1/2$.\\
\indent More generally, by~\cite[Lemma 5.1]{kaltenbacher2022limiting} and a density argument, coercivity assumption \ref{Athree} holds for any kernel that satisfies the following conditions:
\begin{equation}
\begin{aligned}
&\frakKone \in L^1(0,T), \quad  &&(\forall t_0>0)\quad  \frakKone \in W^{1,1}(t_0, T),\\
&\frakKone \geq 0 \ \text{ a.e., } \quad &&(\forall t_0>0) \quad \frakKone'\vert_{[t_0, T]} \leq 0 \ \text{ a.e}.
\end{aligned}
\end{equation}
Thus the analysis in this work holds for all such kernels under the additional condition on their resolvent given in \ref{Atwo}.
\section{Uniform estimates with Westervelt-type nonlinearities} \label{Sec:Est_WestType}
The general strategy in the uniform well-posedness analysis is based on first studying a linearization of the nonlocal equation, and then combining the obtained results with Banach's fixed-point theorem. We take these  two steps in this section. The topic of study here is equation \eqref{JMGT_West_nonlocal_pressure} with Westervelt-type nonlinearities, which we can also rewrite by grouping the nonlinear terms as
\begin{equation}
\taua \frakKone*(\utt-c^2 \Delta u)_t+((1+2 k_1 u )\ut)_t
- c^2\Delta u	- \delta  \Delta \ut= f,
\end{equation}
coupled with initial and boundary data:
$(u, u_t, u_{tt})\vert_{t=0} =(u_0, u_1, u_2)$, $ u_{\vert \partial \Omega}=0$. We introduce a linearization with a variable coefficient,
\begin{equation} \label{GFE_type_eq}
\taua \frakKone *(\utt-c^2\Delta u)_t+ (\aaa(x,t) \ut)_t-c^2\Delta u - \delta \Delta \ut = f,
\end{equation}
where in this section we should understand $\aaa=\aaa(x,t)$ as a placeholder for $1+2k_1 u$. As announced, the main idea in the uniform analysis is to test \eqref{GFE_type_eq} with
 \[y= \utt-c^2\Delta u.\]
An advantage of this combined testing procedure is that we only need the coercivity assumption on $\frakKone$ given in \ref{Athree}. We outline first the main arguments of our energy method, before justifying them rigorously through a Faedo--Galerkin procedure.\\
\indent In terms of assumptions on the variable coefficient $\aaa$, it should be smooth, bounded uniformly in $\tau$, and non-degenerate. More precisely, we assume that 
\begin{equation} \label{smoothness_aaa_West}
	\aaa \in L^\infty(0,T; L^\infty(\Omega)) \cap W^{1,\infty}(0,T; L^4(\Omega))
\end{equation}
and that there exist $\ulaaa$, $\olaaa>0$, independent of $\tau$, such that
\begin{equation} \label{nondeg_aaa}
	\ulaaa < \aaa(x,t) < \olaaa \quad \text{for all} \ (x,t) \in \Omega \times (0,T).
\end{equation} 
Let $f \in L^2(0,T; L^2(\Omega))$. Formally testing the problem with $y(t)=\utt-c^2\Delta u$ and using the coercivity assumption on the kernel gives
\begin{equation}
	\begin{aligned}
	&\inttO	(\aaa \utt-c^2\Delta u - \delta \Delta \ut)(\utt-c^2\Delta u)\dxs\\
	 \leq& \,\begin{multlined}[t]- \inttO \aaa_t \ut (\utt-c^2\Delta u)\dxs+\inttO f (\utt-c^2\Delta u)\dxs\\ + \taua C_{\frakKone} \|\utt(0)-c^2 \Delta u(0)\|^2_{L^2}.
	 \end{multlined}
	\end{aligned}
\end{equation}
From here, using H\"older's and Young's inequalities, for any $\varepsilon>0$ we have
\begin{equation}
	\begin{aligned}
&\begin{multlined}[t]	\|\sqrt{\aaa} \utt\|^2_{L^2(L^2)}+  \frac{\delta}{2}\|\nabla \ut(t)\|^2_{L^2}\Big \vert_0^t+\frac{\delta c^2}{2}\|\Delta u(t)\|^2_{L^2} \Big \vert_0^t
	\end{multlined} \\
\lesssim&\, \begin{multlined}[t] \|\aaa_t\|^2_{L^\infty(L^4)} \|\ut\|^2_{L^2(L^4)} + \|\Delta u\|^2_{L^2(L^2)} + \|f\|^2_{L^2(L^2)}+\varepsilon \|\utt\|^2_{L^2(L^2)} \\
	+ c^2\inttO (\aaa+ 1) \Delta u \utt \dxs
	+ \taua \|u_2\|^2_{L^2} +(1+\taua)\|\Delta u_0\|^2_{L^2}. 
	\end{multlined}
	\end{aligned}
\end{equation}
We can further bound the first term on the right using the embedding $\Honezero \hookrightarrow L^4(\Omega)$:
\begin{equation}
	\|\aaa_t\|^2_{L^\infty(L^4)} \|\ut\|^2_{L^2(L^4)}  \lesssim \|\aaa_t\|^2_{L^\infty(L^4)} \|\nabla \ut\|^2_{L^2(L^2)} .
\end{equation}
The $\aaa+1$ term on the right we can treat by relying on assumption \eqref{nondeg_aaa} as follows:
\begin{equation}
	\begin{aligned}
	\left| \inttO (\aaa+ 1) \Delta u \utt \dxs\right| \lesssim&\, (\olaaa+1) \|\Delta u\|_{L^2(L^2)}\|\utt\|_{L^2(L^2)} \\
	\lesssim&\, \|\Delta u\|_{L^2(L^2)}^2 + \varepsilon \|\utt\|_{L^2(L^2)}^2.
	\end{aligned}
\end{equation}
Thus selecting $\varepsilon>0$ small enough and then applying Gr\"onwall's inequality leads to the following uniform bound in $\tau$:
\begin{equation}
	\begin{aligned}
	&\begin{multlined}[t]	\| \utt\|^2_{L^2(L^2)}+  \|\nabla \ut(t)\|^2_{\Ltwo}+\|\Delta u(t)\|^2_{L^2} 
	\end{multlined} \\
	\leq&\, \begin{multlined}[t] C_{\textup{lin}} \left\{\| u_0\|^2_{H^2}+\|u_1\|_{H^1}^2+\taua \|u_2\|^2_{L^2}+\|f\|^2_{L^2(L^2)}\right\}
	\end{multlined}
	\end{aligned}
\end{equation}
a.e.\ in time. The constant has the form
\begin{equation}\label{C_lin}
	C_{\textup{lin}}=C(\delta)\exp\left\{(1+ \|\aaa_t\|^2_{L^\infty(L^4)})T\right\}
\end{equation}
and tends to $\infty$ as $\delta \searrow 0$. Thus having strong damping $-\delta \Delta \ut$ in the limiting equation is essential for this testing procedure to work. 
\begin{remark}[Initial data for the equation with Westervelt nonlinearities in pressure form]
	If we follow the derivation in Section~\ref{Sec:Modeling} leading to the Westervelt-type equation \eqref{JMGT_West_nonlocal_pressure} in the pressure form when $\frakK \in L^1(0,T)$, the source term is given in \eqref{source_Westervelt_type}. We would need to assume the initial data $(u, u_{tt}) \vert_{t=0}$ to be zero for the analysis in Section~\ref{Sec:Est_WestType} to hold as otherwise we cannot have $L^2(0,T)$ regularity of the right-hand side. However, since equation \eqref{GFE_type_eq} and its treatment are also of independent interest, 
	we consider it below with general initial conditions and source term.
\end{remark}
We formalize next the above reasoning by proving the following existence result for a linear problem which (after also proving uniqueness) we intend to later combine with a fixed-point approach. As we are interested in the limiting behavior as $\tau \searrow 0$, we restrict our considerations to $\tau \in (0, \bar{\tau}]$ for some given fixed $\bar{\tau}>0$.
\begin{proposition}\label{Prop:Wellp_Westlin}
	Let $T>0$, $\delta>0$, and $\tau \in (0, \bar{\tau}]$.
	Let assumptions \ref{Aone} --\ref{Athree} on the kernel $\frakKone$ hold. Let the variable coefficient satisfy \eqref{smoothness_aaa_West} and \eqref{nondeg_aaa}. Let $f \in L^2(0,T; L^2(\Omega))$ and
	\begin{equation}
		\begin{aligned}
		(u_0, u_1, u_2) \in \left(\Honetwo \right) \times \Honezero \times \Ltwo.
		\end{aligned}
	\end{equation}
	Then there exists a solution $u$, such that
	\begin{equation} \label{def_calU_GFE}
		\begin{aligned}
	&u \in \calUW =  L^\infty(0,T; \Honetwo) \cap W^{1, \infty}(0,T; \Honezero) \cap H^2(0,T; \Ltwo), \\ & \taua \frakKone*(\utt-c^2\Delta u)_t \in L^2(0,T; H^{-1}(\Omega)), 
	\end{aligned}
	\end{equation} of the following problem:
	\begin{equation} \label{weak_GFE_lin}
	\left \{	\begin{aligned}
		&\begin{multlined}[t]	\intTO {\tau^a}  \frakKone * (\utt-c^2 \Delta u)_t v \dxt+\intTO (\aaa(x,t) u_{t})_t v\dxt
			\\-c^2 \intTO  \Delta u v \dxt 
			+\delta \intTO\nabla \ut \cdot \nabla v \dxt \\ = \intTO f(x,t) v \dxt,\end{multlined} \\
			&\text{ for all } v \in L^2(0,T; \Honezero)\ \text {with }	(u, u_t, u_{tt})\vert_{t=0} =(u_0, u_1, u_2).
	\end{aligned} \right. 
\end{equation}
 The solution satisfies  
	\begin{equation}\label{final_est_Westlin}
		\begin{aligned}
			\|u\|_{\calUW}^2 \lesssim_T \|u_0\|^2_{H^2}+\|u_1\|^2_{H^1} + \tau^a  \|u_2\|^2_{L^2}+ \|f\|_{L^2(L^2)}^2,
		\end{aligned}
	\end{equation}
where the hidden constant is given by \eqref{C_lin} and does not depend on $\tau.$ 
\end{proposition}
\begin{proof}
We conduct the proof 
using a Faedo--Galerkin semi-discretization in space based on a finite-dimensional subspace $V_n \subset \Honetwo$. We refer to, e.g.,~\cite[Proposition 5.2]{kaltenbacher2022time} for similar arguments in the analysis of equations of higher fractional order. The distinguishing feature of the present Galerkin analysis is that it should be uniform with respect to the relaxation time $\tau$. We present the proof in case $\frakK \in L^1(0,T)$; the arguments given below can be adapted in a straightforward manner to the case $\frakK=\delta_0$. \\
\indent By relying on the existence theory for the Volterra integral equations~\cite{gripenberg1990volterra}, we can prove that there is a unique approximate solution $\un \in W^{2, \infty}(0,T; V_n)$. As these arguments are relatively common, we postpone their details to Appendix~\ref{Appendix:Galerkin}.  \\
\indent Using the estimation techniques outlined at the beginning of this section, we derive the following bound on $\un$:
	\begin{equation} \label{est_discrete_GFE_lin}
		\begin{aligned}
			&\begin{multlined}[t]\	\|\untt\|^2_{L^2(L^2)}+  \|\nabla \unt\|^2_{L^\infty(L^2)}+\|\Delta \un\|^2_{L^\infty(L^2)} 
			\end{multlined} \\
			\lesssim&_T\, \begin{multlined}[t] \| u_0\|^2_{H^2}+\| u_1\|_{H^1}^2+\taua \|u_2\|^2_{L^2}+\|f\|^2_{L^2(L^2)}. 
			\end{multlined}
		\end{aligned}
	\end{equation}
Below we do not relabel any subsequences. Thanks to this bound that is uniform in $n$, there is a subsequence which converges in the following weak(-$*$) sense:
	\begin{equation} \label{weak_limits_Westlin_1}
	\begin{alignedat}{4} 
		\un &\stackrel{\ast}{\relbar\joinrel\rightharpoonup} u && \  &&\text{ in } &&L^\infty(0,T; \Honetwo),  \\
		\unt &\stackrel{\ast}{\relbar\joinrel\rightharpoonup} u_t&&  \  && \text{ in } &&L^\infty(0,T; \Honezero),\\
		\untt &\relbar\joinrel\rightharpoonup u_{tt} && \  &&\text{ in } &&L^2(0,T; \Ltwo),
	\end{alignedat} 
\end{equation} 
as $n \rightarrow \infty$. Since $\frakK \in L^1(0,T)$, then also
	\begin{equation} \label{weak_limits_Westlin_2}
	\begin{alignedat}{4} 
		\frakK* \Delta \unt \stackrel{\ast}{\relbar\joinrel\rightharpoonup} \frakK*\Delta \ut \text{ in } L^\infty(0,T; \Hneg).
	\end{alignedat} 
\end{equation} 
By bootstrapping, we find the following uniform bound on the leading term:
\begin{equation}
	\begin{aligned}
	&\|	{\tau^a}  \frakKone * (\untt-c^2 \Delta \un)_t\|_{L^2(\Hnn)} \\
	=&\, \|-\aaa \untt
	+c^2  \Delta \un +\delta \Delta \unt+f\|_{L^2(\Hnn)} \leq C,
	\end{aligned}
\end{equation}
and thus
\begin{equation}
	\begin{aligned}
		\|	{\tau^a}  \frakKone * \unttt\|_{L^2(\Hnn)}\ \leq  c^2 \bar{\tau}^a	\|	  \frakKone \|_{L^1(0,T)} \| \Delta \unt\|_{L^2(\Hnn)} +C \leq \tilde{C},
	\end{aligned}
\end{equation}
where the constants $C$ and $\tilde{C}$ do not depend either on $n$ nor on $\tau$. Thanks to this uniform bound, by assumption \ref{Atwo} and Lemma~\ref{Lemma:Caputo_seq_compact}, we have 
	\begin{equation} \label{weak_limits_West_3}
	\begin{alignedat}{4} 
	{\tau^a}  \frakKone * \unttt  \relbar\joinrel\rightharpoonup 	{\tau^a}  \frakKone * \uttt   \text{ in } L^2(0,T; \Hneg).
	\end{alignedat} 
\end{equation} 
We can then pass to the limit in the semi-discrete problem in the usual way and show that $u$ solves \eqref{weak_GFE_lin}. Note that since $\tfrakK \in L^2(0,T)$, then from the  bound on 
\[
\taua \frakKone* \unttt := \tilde{f}^{(n)}
\]
in $L^2(0,T; H^{-1}(\Omega))$, using Young's convolution inequality we also have
\begin{equation} \label{Linf_bound_untt}
\begin{aligned}
\taua\|\untt\|_{L^\infty(H^{-1})}=&\, \taua\|u^{(n)}_2 + \tfrakK*\tilde{f}^{(n)}\|_{L^\infty(H^{-1})} \\
\lesssim&\,  \|u^{(n)}_2\|_{H^{-1}} + \|\tfrakK\|_{L^2}\|\tilde{f}^{(n)}\|_{L^2(H^{-1})}.
\end{aligned}
\end{equation}
\indent By the weak limits in \eqref{weak_limits_Westlin_1} and the Aubin--Lions--Simon lemma (see~\cite[Corollary 4]{simon1986compact}), we also have strong convergence in the following sense:
	\begin{equation} \label{weak_limits_GFE_3}
\begin{alignedat}{4} 
\un&\longrightarrow u && \quad \text{ strongly}  &&\text{ in } &&C([0,T]; \Honezero),  \\
\unt &\longrightarrow u_t && \quad \text{ strongly}  &&\text{ in } &&C([0,T]; \Ltwo),
\end{alignedat} 
\end{equation}
as $n \rightarrow \infty$, from which we conclude that $u(0)=u_0$ and $u_t(0)=u_1$.\\
\indent We next show that $u$ also attains the third initial condition. Let $v \in C^1([0,T]; \Honezero)$ with $v(T)=v_t(T)=0$. By subtracting the weak forms for $u$ and $\un$, where we integrate by parts in the leading convolved term using the formula
\begin{equation}\label{diff_conv}
\frakK* w_t = (\frakK*w)_t - \frakK(t) w(0),
\end{equation}
 and then passing to the limit in $n$,  we obtain
\begin{align} \label{attainment_u2}
	-\taua \inttO \frakK(s)(\utt(0)-u_2)v \dxs=0
\end{align}
for all $v \in C^1([0,T]; \Honezero)$ with $v(T)=v_t(T)=0$. Here we have also relied on
$
 \taua(\frakKone* \untt)(0)= \taua(\frakKone* \utt)(0) =0$,
which follows by the $L^\infty$ regularity in time of $\untt$ established \eqref{Linf_bound_untt} and
\begin{equation} \label{Linf_bound_utt}
\begin{aligned}
\taua\|\utt\|_{L^\infty(H^{-1})} \leq &\, \taua \liminf_{n \rightarrow \infty}	\|\untt\|_{L^\infty(H^{-1})} \leq C.
\end{aligned}
\end{equation}
 Therefore from \eqref{attainment_u2} (since $\frakK$ cannot be identically zero by the assumptions on its resolvent), we have $\utt(0)=u_2$. Thus, $u$ is a solution of \eqref{weak_GFE_lin}. By the weak limits in \eqref{weak_limits_Westlin_1} and the weak lower semicontinuity of norms, we conclude that $u$ satisfies stability bound \eqref{final_est_Westlin}.
\end{proof}
Note that $	u \in \calUW$ implies by Lemma 3.3 in~\cite[Ch.\ 2]{temam2012infinite} the following weak continuity in time:
\begin{equation}
	\begin{aligned}
	  u \in C_w([0,T]; \Honetwo), \ \ut \in C_{w}([0,T]; \Honezero).
	\end{aligned}
\end{equation}
We next wish to prove that the solution of the problem we have constructed is the only solution to \eqref{weak_GFE_lin}. To prove uniqueness, we should show that the only solution $u$ of the homogeneous problem (where $f=0$ and $u_0=u_1=u_2$) is $u=0$. However, we are not allowed to test directly with $y=\utt-c^2 \Delta u $ in this setting and replicate the previous energy arguments, as $y$ only belongs to $L^2(0,T; L^2(\Om))$. Instead we employ an approach based on considering an adjoint problem where we adapt the ideas from~\cite[Theorem 3, p.\ 573]{dautray1992evolution} developed for integer-order equations.
\begin{lemma}
The solution $u$ constructed in Proposition~\ref{Prop:Wellp_Westlin} is unique. 
\end{lemma}
\begin{proof}
 The statement will follow by testing the adjoint problem with a convenient test function. Given an arbitrary $g \in L^2(0,T; L^2(\Omega))$, consider the adjoint problem after time reversal: 
\begin{equation}\label{adjoint_pproblem}
	\begin{aligned}
	\begin{multlined}[t]	\taua	\intTO (\frakKone* p)_{ttt}(T-t) v(t)\dxt + \intTO (\tilde{\aaa} p_t)_{t}(T-t) v(t)\dxt \\
		- c^2 \intTO \Delta p(T-t) v \dxt - \taua c^2 \intTO (\frakKone*\Delta p)_t(T-t)v(t)\dxt\\
		- \delta \intTO \nabla p_t(T-t) \cdot \nabla v \dxt = \intTO g(T-t) v \dxt,
		\end{multlined}
	\end{aligned}
\end{equation}
for all $v \in L^2(0,T; \Honezero) $, with $(p, p_t, p_{tt}) \vert_{t=0}=(0,0,0)$ and using the notation $\tilde{\aaa}(t)=\aaa(T-t)$. 
Due to the homogeneous initial data, we have $(\frakKone*p)_{ttt}=\frakKone*p_{ttt}$ and $(\frakKone*\Delta p)_t=\frakKone*\Delta p_t.$ By Proposition~\ref{Prop:Wellp_Westlin}, this problem has a solution $p \in \calUW$ with $\taua \frakK*p_{ttt} \in L^2(0,T; \Hneg)$. \\
\indent We test it next with $v=u$, which is a valid test function since it belongs to $L^2(0,T; \Honezero)$. We use the following integration by parts formula:
\begin{equation} \label{integration_by_parts}
	\begin{aligned}
\int_0^T q_t(T-t)w(t) \dt =\int_0^T q(T-t) w_t(t) \dt- w(T)q(0)+w(0)q(T),
	\end{aligned}
\end{equation}
valid for functions $q$, $w \in W^{1,1}(0,T)$;  see~\cite[Sec.\ 2]{kaltenbacher2021determining}. 
We also rely on the transposition identity (that is, the associativity property of convolution): 
\begin{equation} \label{transposition_identity}
	\begin{aligned}
		\int_0^T (\frakK*q)(T-t) w(t) \dt=\int_0^T (\frakK*w)(t)q(T-t)\dt
	\end{aligned}
\end{equation}
for $q$, $w \in L^1(0,T)$. By \eqref{integration_by_parts}, we have
\begin{equation}
		\begin{aligned}
	\taua	\intTO (\frakK* p)_{ttt}(T-t) u(t)\dxt
	 =\,  \taua	\intTO (\frakK* p_t)(T-t) \utt(t)\dxt,
		\end{aligned} 
\end{equation}
where we have also used that $(u, \ut)\vert_{t=0}=(0,0)$ and \[(\frakK* p)_t(0)=(\frakK* p)_{tt}(0)=0.\]
Then by the associativity property of convolution and the fact that $\utt \vert_{t=0}=0$,
\begin{equation}
	\begin{aligned}
		\taua	\intTO (\frakK* p)_{ttt}(T-t) u(t)\dxt
		=& \,  \taua	\intTO (\frakK* p_t)(T-t) \utt(t)\dxt \\
		=&\,  \taua	\intTO (\frakK* \utt)(t) p_t(T-t)\dxt \\
		=&\,  \taua	\intTO (\frakK* \uttt)(t) p(T-t)\dxt.
	\end{aligned} 
\end{equation}
Next, again by \eqref{integration_by_parts}, 
\begin{equation}
	\begin{aligned}
		\intTO (\tilde{\aaa} p_t)_{t}(T-t) u(t)\dxt =&\, \intTO p(T-t) (\aaa u_{t})_t(t)\dxt. 
	\end{aligned}
\end{equation}
We can treat the other terms on the left-hand side of \eqref{adjoint_pproblem} in a similar manner to arrive at 
\begin{equation}
	\begin{aligned}
	\begin{multlined}[t]	\intTO u(t) g(T-t) \dxt
		 = 
			\taua	\intTO (\frakKone* \uttt)(t) p(T-t)\dxt \\+ \intTO (\aaa \ut )_t(t) p(T-t)\dxt 
		- c^2 \intTO \Delta u(t) p(T-t) \dxt\\ - \taua c^2 \intTO (\frakKone*\Delta \ut)(t)p(T-t)\dxt
		- \delta \intTO \nabla \ut(t) \cdot \nabla p(T-t) \dxt.
		\end{multlined} 
	\end{aligned}
\end{equation}
Since $u$ solves the original (homogeneous) problem, the right-hand side is equal to zero.
As $g$ was arbitrary, from here we conclude that $u=0$.
\end{proof}
To relate the obtained well-posedness result to the nonlinear problem, we next introduce a fixed-point mapping $\calT: \calBW \ni u^* \mapsto u$, 
which maps $u^*$ taken from the ball
\begin{equation}
	\begin{aligned}
		\calBW =\left \{ u \in \calUW:\right.&\, \, \|u\|_{\calUW} \leq R,\ 
		\left	(u, u_t, u_{tt})\vert_{t=0} =(u_0, u_1, u_2) \} \right.
	\end{aligned}
\end{equation}
to the solution $u$ of the linear problem given in \eqref{weak_GFE_lin} with the coefficient $\aaa(u^*) =1+2k_1 u^*$.
The radius $R>0$ is independent of $\tau$ and will be chosen as small as needed by the upcoming proof. 
\begin{theorem}[Uniform well-posedness of equations with Westervelt-type nonlinearities] \label{Thm:WellP_West} 
	Let $T>0$ and $\tau \in (0, \bar{\tau}]$. Assume that $c$, $\delta>0$ and $k_1 \in \R$. Let assumptions \ref{Aone} --\ref{Athree} on the kernel hold. Furthermore, let
	\[
		(u_0, u_1, u_2) \in  \left(\Honetwo \right) \times \Honezero \times  \Ltwo\]
and $f \in L^2(0,T; L^2(\Omega))$. There exists $r=r(T)>0$, independent of $\tau$, such that if
	\begin{equation}\label{smallness_r_Westervelt}
		\|u_0\|^2_{H^2}+\|u_1\|^2_{H^1} + \bar{\tau}^a  \|u_2\|^2_{L^2}+ \|f\|^2_{L^2(L^2)} \leq r^2,
	\end{equation}
	then there is a unique solution $u \in \calBW$ of the nonlinear problem
	\begin{equation} \label{weak_GFE}
	\left \{	\begin{aligned}
		&\begin{multlined}[t]	\intTO \taua  \frakKone * (\utt-c^2 \Delta u)_t v \dxs+\intTO (\aaa(u) u_{t})_t v\dxs
			\\-c^2 \intTO  \Delta u v \dxs 
			+\delta \intTO\nabla \ut \cdot \nabla v \dxs \\=\intTO f v \dxs,\end{multlined} \\
		&\ \text{for all }v \in L^2(0,T; \Honezero),\text{ with }	(u, u_t, u_{tt})\vert_{t=0} =(u_0, u_1, u_2).	
	\end{aligned} \right. 
\end{equation}
 The solutions satisfies the following bound:
\begin{equation}
	\begin{aligned}
		\|u\|_{\calUW} \lesssim_T 	\|u_0\|^2_{H^2}+\|u_1\|^2_{H^1} + \taua  \|u_2\|^2_{L^2}+ \|f\|^2_{L^2(L^2)},
	\end{aligned}
\end{equation}
where the hidden constant does not depend on $\tau$.
\end{theorem}
\begin{proof}
The statement will follow once we check that the conditions of the Banach fixed-point theorem are satisfied for the introduced mapping. We note that the set $\calBW$ is non-empty as the solution of the linear problem with $\aaa=1$ and $f=0$ belongs to it provided the data size $r$ is chosen relative to $R$, so that
\begin{align} \label{cond_r_R_0}
	C_{\textup{lin}}(\|u_0\|^2_{H^2}+\|u_1\|^2_{H^1} + \bar{\tau}^a  \| u_2\|^2_{L^2}) \leq \Clin r^2  \leq R^2.
\end{align}
	\noindent \emph{Self-mapping}.  Let $u^* \in \calB$. Since then $u^* \in \calUW$, the smoothness assumptions on $\aaa$ in Proposition~\ref{Prop:Wellp_Westlin} are satisfied. The non-degeneracy assumption on $\aaa$ is fulfilled by reducing $R>0$. Indeed, we can rely on the embedding $H^2(\Omega) \hookrightarrow L^\infty(\Omega)$ to show that
	\[
	\|2k_1 u^*\|_{L^\infty(L^\infty)} \leq C(\Omega, T)|k_1| R.
	\]
	Then $R$ should be small enough so that
	\begin{align} \label{smallness_R}
	0<  \ulaaa:= 1-C(\Omega, T)|k_1| R \leq \aaa \leq \olaaa:= 1+C(\Omega, T)|k_1| R. 
	\end{align}
	Further, we have the uniform in $\tau$ bound: $\|\aaa_t\|_{L^\infty(L^4)} \lesssim 	|k_1|\|u^*_t\|_{L^\infty(H^1)} \lesssim R$. 
By employing the estimate of Proposition~\ref{Prop:Wellp_Westlin} with the hidden constant given in \eqref{C_lin}, we obtain 
	\begin{equation}
		\begin{aligned}
			\|u\|^2_{\calUW} 
			\leq\,
			Ce^{(1+R^2)T}(\|u_0\|^2_{H^2}+\|u_1\|^2_{H^1} + \bar{\tau}^a  \| u_2\|^2_{L^2}+ \|f\|_{L^{2}(L^2)}^2) 
			\leq\, Ce^{(1+R^2)T}r^2.
		\end{aligned}
	\end{equation}
For sufficiently small $r$, it holds that
	\begin{equation} \label{cond_r_R}
	\begin{aligned}
 Ce^{(1+R^2)T}r^2 \leq R^2.
	\end{aligned}
\end{equation}
Thus, $u \in \calBW$ for $R$ chosen so that \eqref{smallness_R} holds and then $r$ so that \eqref{cond_r_R_0} and \eqref{cond_r_R} hold. Note that these conditions are imposed independently of $\tau$ as all involved estimates are uniform with respect to the relaxation time. \\[1mm]
	\noindent \emph{Strict contractivity}. Let $\calT u^{*} =u$ and $\calT v^*=v$; denote $\bar{\phi}=u-v$ and $\bar{\phi}^*= u^*-v^*$. Then $\phi$ solves  
	\begin{equation}
		\begin{aligned}
			\begin{multlined}[t]	{\tau^a}  \frakKone *( \bar{\phi}_{tt}-c^2 \Delta \bar{\phi})_t+(\aaa(u^*) \bar{\phi}_{t})_t
				-c^2  \Delta \bar{\phi}
				-   \delta   \Delta \bar{\phi}_{t}\end{multlined}
			=\, -2k_1 (\bar{\phi}^* v_{t})_t
		\end{aligned}
	\end{equation}
with homogeneous data. This problem fits the form of the linear problem we have studied in Proposition~\ref{Prop:Wellp_Westlin} with the right-hand side $f= -2k_1 (\bar{\phi}^* v_{t})_t$. Thus using bound \eqref{final_est_Westlin} together with the embeddings  $H_0^1(\Omega) \hookrightarrow L^4(\Omega)$ and $H^2(\Omega) \hookrightarrow L^\infty(\Omega)$ implies 
\begin{equation}
	\begin{aligned}
	\| \bar{\phi}\|_{\calUW} \lesssim&\,  e^{(1+R^2)T} \|-2k_1 \bar{\phi}^* v_{tt} -2k_1 \bar{\phi}^*_t v_t\|_{L^2(L^2)} \\
	\lesssim&\, e^{(1+R^2)T} |k_1| \left \{ \|\bar{\phi}^*\|_{L^\infty(L^\infty)} \|v_{tt}\|_{L^2(L^2)} + \|\bar{\phi}^*_t \|_{L^\infty(L^4)}\|v_t\|_{L^2(L^4)} \right\} \\
		\lesssim&\,  e^{(1+R^2)T} |k_1| R \|\bar{\phi}^*\|_{\calUW},
	\end{aligned}
\end{equation}
from which we obtain strict contractivity in $\|\cdot\|_{\calUW}$ by reducing $R$ (and thus $r$). The statement then follows by Banach's fixed-point theorem as $\calBW$ is closed with respect to $\|\cdot\|_{\calUW}$ . 
\end{proof}
We mention that the constant $|k_1|$ is relatively small in practice for the Westervelt-type equations (it is inversely proportional to the sound of speed squared), which significantly mitigates the smallness assumption imposed on the data.
\section{Limiting behavior of equations with Westervelt-type nonlinearities}\label{Sec:LimWest} 
Equipped with the previous uniform analysis, we are now ready to discuss the limiting behavior of equations with Westervelt-type nonlinearities as $\tau$ vanishes. Again we present the analysis when $\frakKone \in L^1(0,T)$; the arguments can be adapted in a straightforward manner to the case $\frakKone=\delta_0$. Let $\tau \in (0, \bar{\tau}]$. Consider the following initial boundary-value problem:
\begin{equation} \label{IBVP_GFE_tau}
	\left \{	\begin{aligned}
		&\begin{multlined}[t]	{\tau^a}  \frakKone * (\utt^\tau-c^2\Delta u^\tau)_t+(\aaa(u^\tau) \ut^\tau)_t
			-c^2 \Delta u^\tau 
			-\delta   \Delta \ut^\tau=f \ \textup{ in } \Omega \times (0,T),\end{multlined}\\
		&u^\tau\vert_{\partial \Omega}=0, \\
		&	(u^\tau, \ut^\tau, \utt^\tau)\vert_{t=0} =(u^\tau_0, u^\tau_1, u^\tau_2),
	\end{aligned} \right. 
\end{equation}
under the assumptions of Theorem~\ref{Thm:WellP_West} with the uniform (smallness) bound on data:
	\begin{equation}\label{smallness_r_Westervelt_tau}
	\|u^\tau_0\|^2_{H^2}+\|u^\tau_1\|^2_{H^1} + \bar{\tau}^a \|u^\tau_2\|^2_{L^2}+ \|f\|^2_{L^2(L^2)} \leq r^2.
\end{equation}
From the previous analysis and the obtained $\tau$-uniform bounds on the solution, we know that there exists a subsequence, not relabeled, such that 
\begin{equation} \label{weak_limits_tau_West_1}
	\begin{alignedat}{4} 
		u^\tau&\stackrel{\ast}{\relbar\joinrel\rightharpoonup} u && \   &&\text{ in } &&L^\infty(0,T; \Honetwo),  \\
		\ut^\tau& \stackrel{\ast}{\relbar\joinrel\rightharpoonup} u_t&&  \  && \text{ in } &&L^\infty(0,T; \Honezero),\\
		\utt^\tau &\relbar\joinrel\rightharpoonup u_{tt} && \   &&\text{ in } &&L^2(0,T; \Ltwo),
	\end{alignedat} 
\end{equation} 
as $\tau \searrow 0$. Similarly to the techniques used in the existence proof of Proposition~\ref{Prop:Wellp_Westlin}, by the Aubin--Lions--Simon lemma, this further implies
\begin{equation} \label{weak_limits_tau_West_2}
	\begin{alignedat}{4} 
		u^\tau&\longrightarrow u && \ \text{ strongly}  &&\text{ in } &&C([0,T]; \Honezero),  \\
		\ut^\tau &\longrightarrow u_t && \ \text{ strongly}  &&\text{ in } &&C([0,T]; \Ltwo).
	\end{alignedat} 
\end{equation} 
Thus, we have the convergence of initial data as $\tau \searrow 0$ as follows:
\begin{equation} \label{limits_initial_tau_West}
	\begin{alignedat}{4} 
		u^\tau_0=u^\tau(0)&\longrightarrow u(0):=u_0 && \quad \text{ strongly}  &&\text{ in } &&\Honezero,  \\
		u^\tau_1=\ut^\tau(0) &\longrightarrow u_t(0):=u_1 && \quad \text{ strongly}  &&\text{ in } &&\Ltwo.
	\end{alignedat} 
\end{equation} 
We wish to prove that $u$ solves the limiting problem for the Westervelt equation. Let $v \in C^\infty([0,T]; C_0^\infty(\Omega))$ with $v(T)=0$. Setting $\bar{u}=u-u^\tau$ and relying on the weak form in \eqref{weak_GFE} that is satisfied by $u^\tau$, we have
\begin{equation}\label{weak_limit_baru}
	\begin{aligned}
		&\begin{multlined}[t] \intTO (\aaa(u) \ut)_t v \dxt- c^2 \intTO  \Delta u v \dxt + \delta \intTO  \nabla \ut \cdot \nabla v \dxt
		\\	-\intTO f v \dxt \end{multlined}\\
		=&	\begin{multlined}[t]  \intTO ( \aaa(u)\bar{u}_{t})_t v \dxt- c^2 \intTO  \Delta \bar{u} v \dxt+ \delta \intTO  \nabla \bar{u}_{t} \cdot \nabla v \dxt \\
			- \intTO  \taua \frakKone * (\utt^\tau -c^2 \Delta u^\tau)_t v \dxt 
			+2k_1\intTO (\bar{u}\ut^\tau)_t v \dxt.  \end{multlined}
	\end{aligned}
\end{equation}
We should prove that the right-hand side tends to zero as $\tau \searrow 0$. To this end, we exploit  the established weak convergence. By relying on \eqref{weak_limits_tau_West_1} and the equivalence of norms $\|\cdot\|_{L^2}$ and $\|\sqrt{\aaa} \cdot\|_{L^2}$, we conclude that  $\aaa(u) \utt^\tau \relbar\joinrel\rightharpoonup 	\aaa(u) \utt \text{ in } L^2(0,T; \Ltwo)$.
Next, it holds that
\begin{equation}
		\begin{aligned}
\intTO  (\aaa(u))_t\bar{u}_t v \dxt =&\, 2k_1 \intTO  u_t \bar{u}_t v \dxt \\
\lesssim&\,  \|\ut\|_{L^2(L^4)}\|\bar{u}_t\|_{L^\infty(L^2)}\|v\|_{L^2(L^4)},
	\end{aligned}
\end{equation}
and thus this term tends to zero as $\tau \searrow 0$ by the strong convergence in \eqref{weak_limits_tau_West_2}. We can furthermore  conclude that 
\begin{equation}
	\begin{aligned}
	\begin{multlined}[t]	\intTO  (\aaa(u)\bar{u}_{t})_t v \dxt-c^2 \intTO  \Delta \bar{u} v \dxt \\ \hspace*{4cm}+ \delta \intTO  \nabla \bar{u}_{t} \cdot \nabla v \dxt  \rightarrow  0 \quad \text{as } \tau \searrow 0.
		\end{multlined}
	\end{aligned}
\end{equation}
We also have
\begin{equation}
	\begin{aligned}
		&2k_1\intTO (\bar{u}\ut^\tau)_t v \dxt \\
	=&\,2k_1\intTO \bar{u}_t\ut^\tau v \dxt 	+ 2k_1\intTO \bar{u} \utt^\tau v \dxt \\
	 \lesssim&\,  \|\bar{u}_t\|_{L^\infty(L^2)} \|\ut^\tau\|_{L^2(L^4)} \|v\|_{L^2(L^4)}+\|\bar{u}\|_{L^\infty(L^4)} \|\utt^\tau\|_{L^2(L^2)} \|v\|_{L^2(L^4)}\rightarrow 0,	
	\end{aligned}
\end{equation}
thanks to \eqref{weak_limits_tau_West_2} and the embedding $H^1(\Omega) \hookrightarrow L^4(\Omega)$.  It remains to discuss the convolution term on the right-hand side of \eqref{weak_limit_baru}. By noting that
\[
\frakK*\uttt^\tau =( \frakK*\utt^\tau)_t - \frakK(t) u_2^\tau
\]
and since $v(T)=0$, integration by parts yields
\begin{equation}\label{conv_term}
\begin{aligned}
&-\intTO \taua \frakK*\uttt^\tau v \dxt\\
	=&\, \begin{multlined}[t] \taua  \left\{\intTO  \frakKone*\utt^\tau  v_t \dxt \right. \left.+ \inttO  \frakKone(s)u^\tau_2 v\dxt\right\}.
\end{multlined}
\end{aligned}
\end{equation}
The terms in the bracket  in the last line of \eqref{conv_term} are uniformly bounded:
\begin{equation}
\begin{aligned}
&\intTO  \left\{\frakKone*\utt^\tau  v_t +   \frakKone(s)u^\tau_2 v \right\}\dxt \\
\lesssim&\, \|\frakKone\|_{L^1(0,T)} \left\{\|\utt^\tau\|_{L^2(L^2)}\|v_t\|_{L^2(L^2)}+  \|u^\tau_2\|_{L^2(\Omega)}\|v\|_{L^1(L^2)}\right\},
\end{aligned}
\end{equation}
and so the convolution term \eqref{conv_term} also converges to zero as $\tau \searrow 0$.  Similarly,
\[
 \intTO  \taua c^2 \frakKone * \Delta u^\tau_t v \dxs  \rightarrow 0 \quad \text{as} \ \tau \searrow 0.
\]
Therefore, the right-hand side of \eqref{weak_limit_baru} tends to zero as $\tau \searrow 0$ and we conclude that $u$ weakly solves the limiting Westervelt equation. The initial conditions $(u_0, u_1)$ are obtained in the limit of $(u_0^\tau, u_1^\tau)$ by \eqref{limits_initial_tau_West}. \\ 
\indent The limiting problem with $\tau=0$ is  known to be well-posed with $f=0$; 
see~\cite[Theorem 1.1]{meyer2011optimal}. The uniqueness of solutions in a general setting can be obtained by testing the equation satisfied by the difference $\bar{u}$ of two solutions with, for example, $\bar{u}_{t}$.  Note that in the limiting problem one can use a bootstrap argument to show that $\ut \in L^2(0,T; \Honetwo)$. By a subsequence-subsequence argument and the uniqueness of solutions to the limiting problem, we conclude that the whole sequence converges to $u$ as $\tau \searrow 0$, thus arriving at the following result.
\begin{proposition}[Limiting weak behavior of equations with Westevelt nonlinearities] \label{Prop:WeakLim_West}
	Let $T>0$ and $\tau \in (0, \bar{\tau}]$. Let assumptions \ref{Aone}--\ref{Athree} on the kernel hold. Let
\[
	(u^\tau_0, u^\tau_1, u^\tau_2) \in  \left(\Honetwo \right) \times \Honezero \times  \Ltwo\]
and $f \in L^2(0,T; L^2(\Omega))$ with
\begin{equation}
	\|u^\tau_0\|^2_{H^2}+\|u^\tau_1\|^2_{H^1} + \bar{\tau}^a \|u^\tau_2\|^2_{L^2}+\|f\|^2_{L^2(L^2)} \leq r^2,
\end{equation} 
and $r>0$ chosen according to Theorem~\ref{Thm:WellP_West}, independently of $\tau$. Then the family $\{u^\tau\}_{\tau \in (0, \bar{\tau}]}$ of solutions to
	\begin{equation} \label{weak_GFE_tau}
	\left \{	\begin{aligned}
		&\begin{multlined}[t]	\taua \intTO  \frakKone * (\utt^\tau-c^2 \Delta u^\tau)_t v \dxt+\intTO (\aaa(u^\tau) u^\tau_{t})_t v\dxt
			\\-c^2 \intTO \Delta u^\tau v \dxt 
			+\delta \intTO\nabla \ut^\tau \cdot \nabla v \dxt\\= \intTO f  \dxt,\end{multlined} \\
	&	\text{for all }\ v \in L^2(0,T; \Honezero), \, \text{with} \ (u^\tau, u^\tau_t, u^\tau_{tt})\vert_{t=0} =(u^\tau_0, u^\tau_1, u^\tau_2)
	\end{aligned} \right. 
\end{equation}
converges in the sense of \eqref{weak_limits_tau_West_1}, \eqref{limits_initial_tau_West} to the solution \[ u \in \calUW \, \cap H^1(0,T; \Honetwo)\] of the Westervelt equation in pressure form: 
	\begin{equation} \label{weak_Westervelt}
	\left \{	\begin{aligned}
		&\begin{multlined}[t]	\intTO (\aaa(u) \ut)_t v\dxs
			-c^2 \intTO \Delta u v \dxs  
		-\delta \intTO\Delta \ut  v \dxs\\= \intTO f v \dxs,\end{multlined}\\
	&\text{for all}\ v \in L^2(0,T; \Ltwo), \ \text{with} \ (u, \ut)\vert_{t=0} =(u_0, u_1).
	\end{aligned} \right. 
\end{equation}
\end{proposition}
This limiting analysis in $\tau$ can also be seen as an alternative proof of solvability of the Westervelt equation in pressure form for (small) initial data in $\Honetwo \times \Honezero$ and source term in $L^2(0,T; L^2(\Omega))$; the assumptions coincide with the available well-posedness result in~\cite[Theorem 1.1]{meyer2011optimal}.
\subsection{Strong rate of convergence} \label{Subsec:Stronglim_tau_GFEIII} We next wish to prove that the family $\{u^\tau\}_{\tau \in (0, \bar{\tau}]}$ in fact converges strongly at a certain rate in a suitable norm. To simplify matters, we assume in this section that the first two initial conditions are independent of $\tau$; that is, $(u_0^\tau, u_1^\tau)=(u_1, u_2)$. We then note that the difference $\bar{u}=u-u^\tau \in \calUW$ weakly solves 
\begin{equation} \label{ibvp_diff_u}
	\begin{aligned}
(\aaa(u)\bar{u}_{t})_t-c^2\Delta \bar{u} 
			-\delta   \Delta \bar{u}_t=- 2k_1 (\bar{u}\ut^\tau)_t +	{\tau^a}  \frakKone * (\utt^\tau-c^2\Delta u^\tau)_t 
	\end{aligned}
\end{equation}
with homogeneous data. Having in mind the rate of convergence in the standard energy norm (that is, in the space $W^{1,\infty}(0,T; L^2(\Omega)) \cap L^\infty(0,T; \Honezero)$), we could try to test this difference equation with $\bar{u}_t.$ However, the issue arises with the convolution term 
\begin{equation}
	\begin{aligned}
		\inttO \taua \frakKone*\uttt^\tau \bar{u}_t \dxs
	\end{aligned}
\end{equation}
since we only have a uniform bound on $\taua \|\frakKone*\uttt^\tau\|_{L^2(H^{-1})}$ by the previous analysis and not $\|\frakKone*\uttt^\tau\|_{L^2(H^{-1})}$. Integration by parts in time would not help as it would introduce the term 
\revise{$\intO \taua (\frakK* \utt^\tau)(t) \bar{u}_t(t) \dx$ on the right-hand side. Although we could estimate it as follows:
	\[
	\intO \taua (\frakK* \utt^\tau)(t) \bar{u}_t(t) \dx \lesssim \tau^{2a} \|(\frakK*\utt^\tau)(t)\|^2_{L^2}+\|\bar{u}_t(t)\|^2_{L^2},
	\]
in general, we do not have access to a uniform bound on $\tau^\gamma \|\frakK*\utt^\tau\|^2_{L^\infty(L^2)}$ for some non-negative $\gamma <2a$.}
 Thus, although  strong convergence in the energy norm follows by \eqref{weak_limits_tau_West_2}, it does not seem feasible to arrive at a rate of convergence. \\
\indent We adapt here instead the ideas from~\cite{meliani2023} (where linear equations with generalized fractional derivatives of higher order are considered) to obtain strong rate of convergence in a weaker norm. To this end, for $t' \in (0,T)$, we use the following test function:
\begin{equation}\label{test_function}
\begin{aligned}
v(t)= \begin{cases}
\int_t^{t'} \bar{u}(s) \ds \quad &\text{if} \ 0 \leq t \leq t', \\[1mm]
0 \qquad &\text{if} \  t' \leq t \leq T,
\end{cases}
\end{aligned}
\end{equation}
where again $\bar{u}=u-u^\tau$. We refer to~\cite[Ch.\ 7.2]{evans2010partial} and~\cite{kaltenbacher2021determining} for similar ideas employed when proving uniqueness of solutions in the analysis of wave equations. We conveniently have $v_t= -\bar{u}$ if $ 0 \leq t \leq t'$, otherwise $v_t=0$. Further, given a Hilbert space $H$, the following bound holds:
\[
\|v\|_{L^\infty(0,T; H)} \leq \sqrt{T} \|\bar{u}\|_{L^2(0,T; H)}.
\]
Additionally, $v(t')=0$, which is particularly beneficial in the limiting analysis when treating the convolution term. Testing \eqref{ibvp_diff_u} with $v$ defined in \eqref{test_function}, integrating over $(0,T)$, and noting that $\bar{u}(0)=\bar{u}_t(0)=0$, yields
\begin{equation}
\begin{aligned}
&   \int_0^{t'} \intO \aaa(u) \bar{u}_t \bar{u} \dxs  -c^2 \int_0^{t'} \intO \nabla v_t \cdot \nabla v \dxs  +\delta \int_0^{t'}\intO |\nabla \bar{u} |^2 \dxs \\
=&\, \begin{multlined}[t]
	\int_0^{t'} \intO \left\{ -2k_1 (\bar{u}\ut^\tau)_t +\taua \frakKone * (\utt^\tau-c^2\Delta u^\tau)_t  \right\} v \dxs. 
	\end{multlined}
\end{aligned}
\end{equation}
Integration by parts in the first two terms on the left leads to
\begin{equation} \label{est_ubar}
	\begin{aligned}
		&   \frac12 \|\sqrt{\aaa}\bar{u}(t')\|^2_{L^2} + \frac{c^2}{2} \|\nabla v(0)\|^2_{L^2}+ \delta \int_0^{t'}\|\nabla \bar{u} (s)\|_{L^2}^2 \ds \\
		=&\, \begin{multlined}[t]
k_1 \int_0^{t'}\intO \ut \bar{u}^2\dxs-2k_1	\int_0^{t'} \intO (\bar{u}\ut^\tau)_tv \dxs\\\hspace*{5cm}+\taua 	\int_0^{t'} \intO \frakKone * (\utt^\tau-c^2\Delta u^\tau)_t   v \dxs. 
		\end{multlined}
	\end{aligned}
\end{equation}
We next wish to estimate the second and third term on the right-hand side (the first one will be taken care of by Gr\"onwall's inequality). Using integration by parts in time, we have for any $\varepsilon>0$:
\begin{equation}
	\begin{aligned}
		-2k_1\int_0^{t' }\intO   (\bar{u}\ut^\tau)_tv \dxs =&\, 2k_1\int_0^{t' }\intO \bar{u}  \ut^\tau  v_t\dxs \\
		=&\, -2k_1\int_0^{t' }\intO   \ut^\tau \bar{u}^2\dxs   \\
		\lesssim&\,  \|u_t^\tau\|_{L^\infty(L^4)}^2\|\bar{u}\|_{L^2(L^2)}^2 + \varepsilon \|\nabla \bar{u}\|^2_{L^2(L^2)}.
	\end{aligned}
\end{equation}
We recall that $\|u_t^\tau\|_{L^\infty(L^4)} \leq C$, uniformly in $\tau$.  Let us discuss the convolution term. Since the test function is zero at $t'$,  we have. after integration by parts
\begin{equation}
	\begin{aligned}
		 \int_0^{t' }\intO \taua \frakKone*\uttt^\tau v \dxs
			=&\, \, \begin{multlined}[t] \taua  \left\{ 
		\int_0^{t' }\intO  \frakKone*\utt^\tau  \bar{u}\dxs \right.  \left.-\int_0^{t' }\intO  \frakKone(s)u^\tau_2 v \dxs\right\},
	\end{multlined}
	\end{aligned}
\end{equation}
which can be further bounded as follows:
\begin{equation}
\begin{aligned}
	&\taua  \left\{ 
		\int_0^{t' }\intO  \frakKone*\utt^\tau  \bar{u}\dxs \right.  \left.- \int_0^{t' }\intO  \frakKone(s)u^\tau_2 v \dxs\right\}\\
	  \lesssim_T&\, \tau^{2a} \|\frakKone\|^2_{L^1}\|\utt^\tau\|^2_{L^2(L^2)}+\|\bar{u}\|^2_{L^2(L^2)}+ \tau^{2a}\|\frakKone\|^2_{L^1(0,T)} \|u^\tau_2\|_{L^2(\Omega)}^2 .
\end{aligned}
\end{equation}
Additionally,
\[
-\taua c^2	\int_0^{t'} \intO \frakKone * \Delta u^\tau_t   v \dxs 
\lesssim_T \tau^{2a} \|\frakKone\|_{L^1(0,T)}^2\|\nabla \ut^\tau\|_{L^2(L^2)}^2+\varepsilon \|\nabla \bar{u}\|_{L^2(L^2)}^2.
\]
We can use these bounds to further estimate the right-hand side terms in \eqref{est_ubar}. By choosing $\varepsilon>0$ to be sufficiently small, we can absorb the right-hand side $\varepsilon \|\nabla \bar{u}\|^2_{L^2(L^2)}$ terms by the $\delta$ term on the left side of \eqref{est_ubar},  and then employ Gr\"onwall's inequality. Together with the uniform bound 
\[
\|\utt^\tau\|^2_{L^2(L^2)}+\|\nabla \ut^\tau\|_{L^2(L^2)}^2 \leq C,
\]
guaranteed by Theorem~\ref{Thm:WellP_West}, we arrive at the following result.
\begin{theorem}[Limiting strong behavior of equations with Westevelt nonlinearities]\label{Thm:StrongLim_West}
Let the assumptions of Theorem~\ref{Thm:WellP_West} hold with $(u_0^\tau, u_1^\tau)=(u_1, u_2)$ independent of $\tau$. Let $\{u^\tau\}_{\tau \in (0, \bar{\tau}]}$ be the family of solutions to \eqref{weak_GFE_tau}
and let $u$ be the solution of the corresponding limiting problem for the Westervelt equation with $\tau=0$. Then there exists $C>0$, independent of $\tau$, such that
\begin{equation}
\|u-u^\tau\|_{L^\infty(L^2)}  +  \|\nabla (u-u^\tau)\|_{L^2(L^2)} \leq  C \taua.
\end{equation}
\end{theorem}
\noindent This theorem reveals that the nonlocal equation 
\begin{equation}
\taua \frakKone *(\utt-c^2\Delta u)_t+ ((1+2k_1 u) \ut)_t-c^2\Delta u - \delta \Delta \ut  = f
\end{equation}
can be seen as an approximation of the strongly damped Westervelt equation for small enough $\tau$, under the assumptions on the kernel made in \ref{Aone}--\ref{Athree}. More precisely, solutions of the nonlocal problem converge to the solutions of the limiting problem with the order $a$ in the norm of the space $L^\infty(0,T; L^2(\Omega)) \cap L^2(0,T; \Honezero)$.
 \begin{remark}[On the convergence in the energy norm]
 \revise{	The main obstacle to obtaining convergence in the standard energy norm is the lack of a uniform bound on $\tau^\gamma \|\frakK*\utt\|^2_{L^\infty(L^2)}$ for any $0 \leq \gamma <2a$ with respect to the relaxation time.} \revise{Note that the situation significantly simplifies in the integer-order case with $\frakK=\delta_0$, where a uniform bound on $\tau \|\utt\|^2_{L^\infty(L^2)}$ can be deduced from the analysis in Section~\ref{Sec:Est_WestType}. This setting is already covered by the results of~\cite{bongarti2020vanishing}. \\
 	\indent In general,} an idea might be to uniformly bound $\|\utt\|_{L^\infty(L^2)}$. To this end, one could differentiate equation \eqref{GFE_type_eq} and perform an analogous testing procedure to before by testing it with $y_t=(\utt-c^2 \Delta u)_t$ (thereby paying the price of stronger regularity and smallness assumptions on the data). However, the issue is that we would need to ensure the boundedness of $\uttt(0)$ to write the leading term of the time-differentiated equation in the form of $\frakKone*(u_{ttt}-c^2 \Delta \ut)_t$ suitable for such testing. 
 \end{remark}
\section{Uniform estimates with Kuznetsov--Blackstock-type nonlinearities} \label{Sec:EstBK}  The ideas put forward in the previous sections can be extended to work for the Kuznetsov--Blackstock nonlinearites under stronger assumptions on data. Again, we first outline the key ideas before formalizing them. In this section, the linearized equation has the form
\begin{equation} \label{linear_eq_higher}
	\taua \frakKone *(\utt-c^2\Delta u)_t+ \aaa(x,t) \utt-c^2\bbb(x,t)\Delta u - \delta \Delta \ut = \calF(x,t)
\end{equation}
where we think of the coefficient $\aaa$ as a placeholder for $1+2k_1 \ut$, the coefficient $\bbb$ for $1-2k_2 \ut$, and the source term $\calF$ for $-2 k_3 \nabla u\cdot \nabla \ut +f$. 
To eventually treat Kuznetsov--Blackstock nonlinearities, we need more smoothness of the solution compared to before. Therefore, here we test the linearized equation (in a smooth semi-discrete setting) 
 with \[-\Delta y= -\Delta (\utt-c^2\Delta u).\]
  Let $\calF \in L^2(0,T; H^1(\Omega))$. 
 By using again the coercivity of the kernel in \ref{Athree}, this approach leads to the following estimate:
\begin{equation} \label{higher_testing_GFE}
	\begin{aligned}
		&\inttO	\nabla( \aaa \utt-c^2\bbb\Delta u - \delta \Delta \ut) \cdot \nabla (\utt-c^2\Delta u)\dxs\\
		\leq& \, \intt \| \calF(s) \|_{H^1} \|\utt(s)-c^2\Delta u(s)\|_{H^1}\ds+ \taua C_{\frakKone} \|\nabla \utt(0)-c^2 \nabla \Delta u(0)\|^2_{L^2}.
	\end{aligned}
\end{equation}
Above we have relied on the fact that $\utt=\Delta u= \Delta \ut=0$ on the boundary in the semi-discrete setting, provided the discretization is based on the smooth eigenfunctions of the Dirichlet--Laplacian operator. We have also used the trace theorem to treat the $\calF$ term a.e.\ in time:
\begin{equation}
\begin{aligned}
\left|-\intO \calF \Delta y \dxs \right| 
\leq&\, \|\nabla \calF\|_{L^2} \|\nabla y\|_{L^2}+\left\| \frac{\partial y}{\partial n}\right\|_{H^{-1/2}(\partial \Omega)}\|\calF\|_{H^{1/2}(\partial \Omega)} \\
\lesssim&\, \|\calF\|_{H^1}\|y\|_{H^1}.
\end{aligned}
\end{equation}
 Starting from \eqref{higher_testing_GFE} and transferring the $\bbb$ terms to the right side, we then further have
\begin{equation}
	\begin{aligned}
&\begin{multlined}[t]\int_0^t\|\sqrt{\aaa(s)} \nabla \utt(s)\|^2_{L^2} \ds+\frac{\delta c^2}{2}\|\nabla \Delta u(s)\|^2_{L^2}\Big \vert_0^t+\frac{\delta}{2}\|\Delta u_t(s)\|^2_{L^2} \Big \vert_0^t \end{multlined} \\
\leq&\, \begin{multlined}[t]-\inttO \utt \nabla \aaa \cdot \nabla \utt \dxs  +c^2 \inttO (\aaa \nabla  \utt+\utt \nabla \aaa) \cdot \nabla \Delta u \dxs\\\hspace*{-4mm}-c^4 \inttO\left( \bbb |\nabla \Delta u|^2+ \Delta u \nabla \bbb  \cdot \nabla \Delta u \right)\dxs\\+c^2 \inttO (\bbb \nabla \Delta u + \Delta u \nabla \bbb)\cdot \nabla \utt \dxs 
\\+\intt \| \calF(s) \|_{H^1} \|\utt(s)-c^2\Delta u(s)\|_{H^1}\ds 
+ \taua C_{\frakKone} \|\nabla \utt(0)-c^2 \nabla \Delta u(0)\|^2_{L^2}.\hspace*{10mm} \end{multlined}
	\end{aligned}
\end{equation}
Besides sufficient smoothness of the coefficients $\aaa$ and $\bbb$, and the non-degeneracy of $\aaa$ as before, to arrive at a uniform bound here we also need smallness of the coefficient $\aaa$; that is, we assume that
\begin{equation} \label{smallness_m}
	\begin{aligned}
		\|\nabla \aaa\|_{L^\infty(L^4)} \leq m.
	\end{aligned}
\end{equation}
Using also the embedding $\Honezero \hookrightarrow L^4(\Omega)$, if $m>0$ is small enough, the following term can be absorbed by the left-hand side:
\begin{equation}
	\begin{aligned}
	\inttO \utt \nabla \aaa \cdot \nabla \utt \dxs \lesssim \| \nabla \aaa\|_{L^\infty(L^4)} \| \nabla \utt\|_{L^2(L^2)}^2 \lesssim m \| \nabla \utt\|_{L^2(L^2)}^2.
	\end{aligned}
\end{equation}
Note that in the case of a linear equation with constant coefficients, condition \eqref{smallness_m} trivially holds. The other terms can be treated by first H\"older's inequality and then Young's and Gr\"onwall's inequalities. In this manner, we obtain
\begin{equation} \label{higher_order_est}
	\begin{aligned}
	&\begin{multlined}[t] \|\nabla \Delta u(t)\|^2_{L^2}+\|\Delta u_t(t)\|^2_{L^2}  + \int_0^t\| \nabla \utt\|^2_{L^2}\ds\end{multlined} \\
	\lesssim_T&\, \begin{multlined}[t] \|\nabla \Delta u_0\|^2_{L^2}+\|\Delta u_1\|^2_{L^2} +\taua \|\nabla u_2\|^2_{L^2} + \| \calF\|^2_{L^2(H^1)}. \end{multlined}
	\end{aligned}
\end{equation}
The  hidden constant has the form
\begin{equation}\label{C_lin_BK}
	C_{\textup{lin}}=C(\delta)\exp\left\{(1+ \|\bbb\|^2_{L^\infty(L^\infty)}+\|\nabla \bbb\|^2_{L^2(L^4)})T\right\}
\end{equation}
and tends to $\infty$ as $\delta \searrow 0$.  It is clear that the ($\tau$-independent) solution space should now be
	\begin{equation} \label{def_calUBK}
		\begin{aligned}
	\begin{multlined}[t] \calUBK = L^\infty(0,T; \Honethree)  \cap W^{1, \infty}(0,T; \Honetwo) \\ \cap H^2(0,T; \Honezero).
		\end{multlined}
	 \end{aligned}
\end{equation} 
We formalize the above arguments with the following well-posedness result. 
\begin{proposition}\label{Prop:Wellp_GFE_lin_higher}
	Let $T>0$, $\delta>0$, and $\tau \in (0, \bar{\tau}]$. Let  assumptions \ref{Aone}--\ref{Athree} on the kernel hold. Let the coefficients $\aaa$ and $\bbb$ satisfy 
\begin{equation}
\begin{aligned}
	&\aaa \in  L^\infty(0,T; L^\infty(\Omega) \cap W^{1,4}(\Omega)),\\
	&\bbb \in  L^\infty(0,T; L^\infty(\Omega)) \cap L^2(0,T; W^{1,4}(\Omega)).
\end{aligned}
\end{equation}
Assume that $\aaa$ does not degenerate so that \eqref{nondeg_aaa} holds. Let also \[ \calF \in L^2(0,T; H^1(\Omega))\] and
\[	(u_0, u_1, u_2) \in \Honethree \times \left( \Honetwo \right) \times \Honezero.\]
where
	\begin{equation}
	\Honethree = \{u \in  \Hthree: \, u_{\vert \partial \Omega}= \Delta u_{\vert \partial \Omega} =0\}.
\end{equation}
	Then there exists $m>0$, independent of $\tau$, such that if the coefficient $\aaa$ satisfies condition \eqref{smallness_m}, there is a unique solution 
	\[ u \in \calUBK, \quad \taua  \frakKone * (\utt-c^2 \Delta u)_t \in L^2(0,T; L^2(\Omega))\] of the following problem:
	\begin{equation} \label{IBVP_Westtype_lin_higher}
	\begin{aligned}
			\taua  \frakKone * (\utt-c^2 \Delta u)_t+\aaa(x,t) u_{tt}
			-c^2 \bbb(x,t) \Delta u -\delta \Delta \ut= \calF
	\end{aligned}
\end{equation} 
with $u\vert_{\partial \Omega}=0$ and $(u, u_t, u_{tt})\vert_{t=0} =(u_0, u_1, u_2)$. The solutions satisfies 
	\begin{equation}\label{final_est_GFE_lin}
		\begin{aligned}
			\|u\|_{\calUBK}^2 \lesssim_T \|u_0\|^2_{H^3}+\|u_1\|^2_{H^2} + \tau^a  \|u_2\|^2_{H^1}+ \|\calF\|_{L^2(H^1)}^2,
		\end{aligned}
	\end{equation}
	where the hidden constant has the form given in \eqref{C_lin_BK} and does not depend on $\tau.$ 
\end{proposition}
\begin{proof}
The core of the arguments needed to prove the statement is contained in the above energy analysis leading up to \eqref{higher_order_est} which can be made rigorous through the Faedo--Galerkin procedure as before. We omit these details here. By boostrapping, we also obtain a $\tau$-uniform bound on $\taua	\|\frakKone * (\utt-c^2 \Delta u)_t\|_{L^2(L^2)} \leq C$.
 We note that the uniqueness of the constructed solution may be shown by testing the homogeneous problem by $y= \utt-\Delta u$, which is a valid test function in this setting.
\end{proof}
To connect this result to the nonlinear problem, we employ again the Banach fixed-point theorem, this time  to the mapping $\calT: \calB^{\textup{KB}} \ni u^* \mapsto u$, 
where $u$ solves \eqref{IBVP_Westtype_lin_higher} with
\begin{equation}
	\begin{aligned}
		\aaa(u^*_t) =&\,1+2k_1 u_t^*, \quad \bbb(u^*_t)  = 1-2k_2 u_t^*, \\
		\calF=& \,  - \calN(\nabla u^*, \nabla u^*_t)+ f= - 2k_3 \nabla u^*\cdot \nabla u_t^*+ f
	\end{aligned}
\end{equation}
and the same data, and  the previous fixed-point iterate $u^*$ is taken from the ball
\begin{equation} \label{ball_KB}
	\begin{aligned}
		\calB^{\textup{KB}} =\left \{ u \in \calUBK:\right.&\, \, \|u\|_{\calUBK} \leq R,\
		\left	(u, u_t, u_{tt})\vert_{t=0} =(u_0, u_1, u_2) \}. \right.
	\end{aligned}
\end{equation}
\vspace*{-6mm}
\begin{theorem}[Uniform well-posedness of equations with Kuznetsov--Blackstock nonlinearities]\label{Thm:WellP_GFE_BlckstockKuznetov} 
	Let $T>0$ and $\tau \in (0, \bar{\tau}]$. Let $c$, $\delta>0$ and $k_{1,2,3} \in \R$. Let assumptions \ref{Aone} and \ref{Atwo} on the kernel hold. Furthermore, let \[(u_0, u_1, u_2) \in  \Honethree \times \left( \Honetwo \right) \times  \Honezero\]
	and $f \in L^2(0,T; H^1(\Omega))$. There exists $r=r(T)>0$, independent of $\tau$, such that if
	\begin{equation} 
		\|u_0\|^2_{H^3}+\|u_1\|^2_{H^2} + \bar{\tau}^a  \|u_2\|^2_{H^1}+ \|f\|^2_{L^2(H^1)} \leq r^2,
	\end{equation}
	then there is a unique solution $u \in 	\calB^{\textup{KB}}$ of the nonlinear problem, given by
	\begin{equation} 
	\begin{aligned}
		&\begin{multlined}[t]	{\tau^a}  \frakKone * (\utt-c^2 \Delta u)_t+(1+2k_1 u_t) u_{tt}
			-c^2 (1-2k_2 u_t)\Delta u -\delta \Delta \ut\\ \hspace*{4.5cm}+ 2 k_3 \nabla u \cdot \nabla u_t=f  \end{multlined}
	\end{aligned} 
\end{equation}
with $u\vert_{\partial \Omega}=0$ and $(u, u_t, u_{tt})\vert_{t=0} =(u_0, u_1, u_2)$. The solutions satisfies 
\begin{equation}\label{final_est_GFE}
	\begin{aligned}
		\|u\|_{\calUBK}^2 \lesssim_T \|u_0\|^2_{H^3}+\|u_1\|^2_{H^2} + \tau^a  \|u_2\|^2_{H^1}+ \|f\|_{L^2(H^1)}^2. 
	\end{aligned}
\end{equation}
\end{theorem}
\begin{proof}
The proof can be conducted similarly to the proof of Theorem~\ref{Thm:WellP_West} using the Banach fixed-point theorem; we only point out the main differences here. The smallness of the coefficient $\aaa$ can be guaranteed by observing that
\begin{equation}
\begin{aligned}
	\|\nabla \aaa\|_{L^\infty(L^4)} \leq 2|k_1|	\|\nabla u_t^*\|_{L^\infty(L^4)} 
	 \lesssim |k_1|	\|u_t^*\|_{L^\infty(H^2)}  \lesssim |k_1|	 R
\end{aligned}
\end{equation}
and taking $R$ (independently of $\tau$) small enough so that the right-hand side is smaller than $m$. The self-mapping property can be obtained similarly to before for small $r$ and $R$ by noting that
	\begin{equation}
		\begin{aligned}
		\|	\calF\|_{L^2(H^1)}
			\lesssim \begin{multlined}[t] \|u^*\|_{L^\infty(H^3)} \|  u_t^*\|_{L^2(H^2)}+ \|f\|_{L^2(H^1)}
			\end{multlined}
		\lesssim\, R^2+ \|f\|_{L^2(H^1)}
		\end{aligned}
	\end{equation}
and thus
$		\|u\|^2_{\calUBK} 
		\leq\,
		Ce^{(1+R^2)T}(r^2+ R^4)$.
To discuss contractivity, let $\calT u^{*} =u$ and $\calT v^*=v$. We denote their differences by $\bar{\phi}=u-v$ and $\bar{\phi}^*= u^*-v^*$. Then $\phi$ is a solution of the following equation:
\begin{equation}\label{diff_eq_contractivity_BK}
	\begin{aligned}
		{\tau^a}  \frakKone *( \bar{\phi}_{tt}-c^2\Delta \bar{\phi})_t+\aaa(u_t^*) \bar{\phi}_{t}
			-c^2   \bbb(u^*_t) \Delta \bar{\phi}
			-   \delta   \Delta \bar{\phi}_{t}= \mathcal{\tilde{F}}
	\end{aligned}
\end{equation}
with the right-hand side
\begin{equation}
	\begin{aligned}
\mathcal{\tilde{F}}
		=&\, -2k_1 \bar{\phi}_t^* v_{tt} -2k_2c^2 \bar{\phi}^*_t\Delta v +k_3 \nabla \bar{\phi}^* \cdot(\nabla u_t^*+\nabla v_t^*)+k_3 \nabla \bar{\phi}^*_t \cdot (\nabla u^*+\nabla v^*)
	\end{aligned}
\end{equation}
and homogeneous boundary and initial conditions. It is straightforward to check that 
\begin{equation}
	\begin{aligned}
		\| \mathcal{\tilde{F}}\|_{L^2(H^1)} \lesssim  R \|\bar{\phi}^*\|_{\calUBK}.
	\end{aligned}
\end{equation}
Thus employing energy estimate \eqref{final_est_GFE_lin} for the solution of \eqref{diff_eq_contractivity_BK} gives
\begin{equation}
	\begin{aligned}
	 \|\bar{\phi}\|_{\calUBK} \lesssim   e^{(1+ R^2)T} R \|\bar{\phi}^*\|_{\calUBK},
	\end{aligned}
\end{equation}
and we can obtain the strict contractivity of the mapping $\calT$ by additionally reducing the radius $R>0$. An application of Banach's fixed-point theorem yields the desired result.
\end{proof}
\section{Limiting behavior of equations with Kuznetsov--Blackstock-type nonlinearities} \label{Sec:Lim_BK}
It remains to discuss the limiting behavior of equations with Kuznetsov--Blackstock nonlinearities as $\tau \searrow 0$. Under the assumptions of Theorem~\ref{Thm:WellP_GFE_BlckstockKuznetov} with data uniformly bounded in $\tau$, such that
\[
 	\|u^\tau_0\|^2_{H^3}+\|u^\tau_1\|^2_{H^2} + \bar{\tau}^a  \|u^\tau_2\|^2_{H^1}+ \|f\|^2_{L^2(H^1)} \leq r^2,\]
we investigate in this section the vanishing thermal relaxation limit of the family $\{u^\tau\}_{\tau \in (0, \bar{\tau}]}$ of solutions to the following problem:
\begin{equation} \label{IBVP_GFE_tau_BlackstockKuznetsov}
\left \{	\begin{aligned}
&\begin{multlined}[t]	{\tau^a}  \frakKone * (\utt^\tau-c^2\Delta u^\tau)_t+\aaa(u^\tau_t) \utt^\tau
-c^2 \bbb(u^\tau_t) \Delta u^\tau 
-\delta   \Delta \ut^\tau\\ \hspace*{3.4cm}+\calN(\nabla u^\tau, \nabla u^\tau_t)=f \qquad \textup{ in }  \Omega \times (0,T),\end{multlined}\\
&u^\tau\vert_{\partial \Omega}=0, \\
&	(u^\tau, \ut^\tau, \utt^\tau)\vert_{t=0} =(u^\tau_0, u^\tau_1, u^\tau_2).
\end{aligned} \right. 
\end{equation}
\indent We can adapt the arguments from Section~\ref{Sec:LimWest}  to prove the weak convergence of this family to the solution of the Kuznetsov--Blackstock equation:
\begin{equation}\label{BK_eq}
	(1+2k_1 \ut) \utt
	-c^2 (1+2k_2 u_t) \Delta u
	-\delta   \Delta \ut+2k_3 \nabla u \cdot \nabla u_t=f
\end{equation}
as the relaxation time tends to zero. Indeed, by Theorem~\ref{Thm:WellP_GFE_BlckstockKuznetov}  and the obtained uniform bounds, we conclude that there is a subsequence, again not relabeled, such that
\begin{equation} \label{weak_limits_tau_GFE_BK}
\begin{alignedat}{4} 
u^\tau&\stackrel{\ast}{\relbar\joinrel\rightharpoonup} u &&\text{ in } &&L^\infty(0,T; \Honethree),  \\
\ut^\tau &\stackrel{\ast}{\relbar\joinrel\rightharpoonup}u_t && \text{ in } &&L^\infty(0,T; \Honetwo),\\
\utt^\tau &{\relbar\joinrel\rightharpoonup} u_{tt} && \text{ in } &&L^2(0,T; \Honezero),
\end{alignedat} 
\end{equation}
as $\tau \searrow 0$. Additionally, by the Aubin--Lions--Simon lemma, we have
\begin{equation}\label{weak_limits_tau_2_BK}
\begin{alignedat}{4} 
u^\tau&\longrightarrow u && \quad \text{ strongly}  &&\text{ in } &&C([0,T]; \Honetwo),  \\
\ut^\tau &\longrightarrow u_t && \quad \text{ strongly}  &&\text{ in } &&C([0,T]; \Honezero),
\end{alignedat} 
\end{equation} 
and thus the sequence of initial data converges in the following sense:
\begin{equation} \label{limits_tau_initial_BK} 
\begin{alignedat}{5} 
u_0^\tau=u^\tau(0)&\longrightarrow u(0)&&:=u_0 && \quad \text{ strongly}  &&\text{ in } && \Honetwo,  \\
u_1^\tau=\ut^\tau &\longrightarrow u_t(0)&&:=u_1 && \quad \text{ strongly}  &&\text{ in } && \Honezero
\end{alignedat} 
\end{equation} 
as $\tau \searrow 0$. It remains to prove that $u$ is a unique solution of the limiting problem. 
\begin{proposition}[Limiting weak behavior of equations with Kuznetsov--Blackstock nonlinearities] \label{Prop:WeakLim_BK} 
	Let the assumptions of Theorem~\ref{Thm:WellP_GFE_BlckstockKuznetov} hold.  Then the family $\{u^\tau\}_{\tau \in (0, \bar{\tau}]}$ of solutions to \eqref{IBVP_GFE_tau_BlackstockKuznetsov} converges weakly in the sense of \eqref{weak_limits_tau_GFE_BK}, \eqref{limits_tau_initial_BK} to the solution of the corresponding limiting problem \[u \in \calUBK \cap H^1(0,T; \Honethree)\] for the Kuznetsov--Blackstock equation \eqref{BK_eq} with homogeneous Dirichlet conditions and $(u, \ut)\vert_{t=0}=(u_0, u_1)$. 
\end{proposition}
\begin{proof}
The proof follows using analogous arguments to those in Section~\ref{Sec:LimWest} by proving that $u$ in \eqref{weak_limits_tau_GFE_BK} solves the limiting problem. The main difference compared to the similar analysis in Proposition~\ref{Prop:WeakLim_West}  comes from treating the nonlinear terms. These can be tackled by exploiting the weak limits in \eqref{weak_limits_tau_GFE_BK} together with the strong convergence in \eqref{weak_limits_tau_2_BK}.
We omit the details here.  In the limiting problem, one can use a bootstrap argument to show that additionally $\ut \in L^2(0,T; \Honethree)$. Uniqueness of solutions to the limiting problem can be shown by testing the equation satisfied by the difference $\bar{u}$ of two solutions with $\bar{u}_{t}$; similar ideas can be found, for example, in~\cite[Theorem 5.1]{kaltenbacher2022limiting}.  Thus by a subsequence-subsequence argument we conclude that the whole sequence $\{u^\tau\}_{\tau \in (0, \bar{\tau}]}$ converges to $u$ in the sense of \eqref{weak_limits_tau_GFE_BK}. 
\end{proof}
Proposition~\ref{Prop:WeakLim_BK} covers the case  $\frakKone=\delta_0$ and thus provides weak convergence of solutions of the third-order Jordan--Moore--Gibson--Thompson equation with Kuznetsov--Blackstock nonlinearities. 
This result generalizes~\cite[Theorem 7.1]{kaltenbacher2019jordan}, where Kuznetsov-type nonlinearities (that is, equations with $k_2=0$) have been considered under the same assumptions on the data.
\subsection{Strong rate of convergence with Blackstock-type nonlinearities} To discuss the strong convergence (and determine the rate), we see the difference $\bar{u}=u-u^\tau$ as the solution of
\begin{equation} \label{ibvp_diff_u_BK}
	\begin{aligned}
	\begin{multlined}[t]	\aaa(\ut)\bar{u}_{tt}-c^2 \bbb(\ut) \Delta \bar{u} 
		-\delta   \Delta \bar{u}_t=-2k_1 \bar{u}_t \utt^\tau-2 c^2 k_2 \bar{u}_t \Delta u^\tau-2k_3 \nabla \bar{u} \cdot \nabla u^\tau_t\\
		-2k_3\nabla u \cdot \nabla\bar{u}_t  +	{\tau^a}  \frakKone * (\utt^\tau-c^2\Delta u^\tau)_t. 
	\end{multlined}	
	\end{aligned}
\end{equation}
To simplify matters, we assume $(u_0^\tau, u_1^\tau)=(u_0, u_1)$ to be independent of $\tau$ in this section, so that $\bar{u}$ satisfies homogeneous initial conditions. Here obtaining strong convergence of solutions as $\tau \searrow 0$ with the order $a$ does not seem feasible using the procedure from before with Westervelt-type nonlinearities since, after testing with $v$ defined in \eqref{test_function}, we would have to further treat the term $-2k_1\inttO \bar{u}_t \utt^\tau v \dxs$
by integration by parts. 
 This would result in the third time derivative $\uttt^\tau$ which we cannot control. Testing with $\bar{u}_t$, on the other hand, would lead to the same issues related to the convolution term as in Section~\ref{Sec:LimWest}.\\
 \indent For this reason we restrict the discussion in this section to the Blackstock-type nonlinearities; that is, we assume that $k_1=0$ and thus $\aaa \equiv 1$ so that in the limit $\tau \searrow 0$ we obtain the Blackstock wave equation. The difference equation above then simplifies to
\begin{equation} \label{ibvp_diff_u_B}
\begin{aligned}
\begin{multlined}[t]	\bar{u}_{tt}-c^2 \bbb(\ut) \Delta \bar{u} 
-\delta   \Delta \bar{u}_t=-2 c^2 k_2 \bar{u}_t \Delta u^\tau
-2k_3 \nabla\bar{u}_t \cdot \nabla u-2k_3 \nabla u^\tau_t \cdot \nabla \bar{u} \\+	{\tau^a}  \frakKone * (\utt^\tau-c^2\Delta u^\tau)_t.
\end{multlined}	
\end{aligned}
\end{equation}
In this case ($k_1=0$), we can even test with $- \Delta v$; recall that $v$ is defined in \eqref{test_function}. Similarly to \eqref{est_ubar}, integrating over $\Omega$ and $(0,t')$ then leads to 
\begin{equation}\label{v_testing_Blackstock}
	\begin{aligned}
		&   \|\nabla \bar{u}(t')\|^2_{L^2}  + \frac{c^2}{2} \|\Delta v(0)\|^2_{L^2}  +\delta \int_0^{t'}\intO |\Delta \bar{u} |^2 \dxs \\
		=&\, \begin{multlined}[t] 2c^2 k_2 \int_0^{t'} \intO \ut \Delta \bar{u} \Delta v \dxs-	\taua\int_0^{t'} \intO  \frakKone * (\utt^\tau-c^2\Delta u^\tau)_t   \Delta v \dxs\\+ \int_0^{t'}\intO \left\{2 c^2 k_2 \bar{u}_t \Delta u^\tau
			+2k_3 \nabla\bar{u}_t \cdot \nabla u+2k_3 \nabla u^\tau_t \cdot \nabla \bar{u} \right\} \Delta v \dxs
			. 
		\end{multlined}
	\end{aligned}
\end{equation}
We can treat the convolution term analogously to before in Theorem~\ref{Thm:StrongLim_West} so we discuss the remaining terms. First, using H\"older's inequality and the embedding $H^2(\Omega) \hookrightarrow L^\infty(\Omega)$ yields
\begin{equation}
	\begin{aligned}
	2c^2 k_2 \int_0^{t'} \intO \ut \Delta \bar{u} \Delta v \dxs
	  \lesssim_T&\, \|\Delta \ut\|_{L^2(L^2)}   \|\Delta \bar{u}\|^2_{L^2(L^2)},
	\end{aligned}
\end{equation}
where we have also relied on the inequality $\|\Delta v\|_{L^\infty(L^2)} \leq \sqrt{T} \|\Delta \bar{u}\|_{L^2(L^2)}$.
Next, by using integration by parts in time and recalling that $v_t= - \bar{u}$ and $v(t')=0$, we find
\begin{equation}
	\begin{aligned}
	2c^2 k_2	\int_0^{t'} \intO  \bar{u}_t \Delta u^\tau \Delta v \dxs 
		 \lesssim&\,\left\{ \sqrt{T}\|\Delta \ut^\tau\|_{L^2(L^2)} +\|\Delta u^\tau\|_{L^\infty(L^2)} \right\}\|\Delta \bar{u}\|^2_{L^2(L^2)},
	\end{aligned}
\end{equation}
where we have also again relied on the embedding $H^2(\Omega) \hookrightarrow L^\infty(\Omega)$. 
We can treat the $k_3$ terms using again integration by parts in time, H\"older's inequality, and the embedding $H^1(\Omega) \hookrightarrow L^4(\Omega)$:
\begin{equation}
	\begin{aligned}
		&\int_0^{t'} \intO\left\{2k_3 \nabla u^\tau_t \cdot \nabla \bar{u}+	2k_3 \nabla\bar{u}_t \cdot \nabla u \right\} \Delta v \dxs \\
		\lesssim_T&\, \begin{multlined}[t] \left\{\|\Delta u_t^\tau\|_{L^2(L^2)}+\|\Delta u_t\|_{L^2(L^2)}
			 \right\}\|\nabla \bar{u}\|_{L^2(L^2)}^2+ \|\Delta u\|_{L^\infty(L^4)}^2\|\nabla \bar{u}\|^2_{L^2(L^2)}\\+\varepsilon \|\Delta \bar{u}\|^2_{L^2(L^2)}.
		\end{multlined}
	\end{aligned}
\end{equation}
To absorb the arising $ \|\Delta \bar{u}\|^2_{L^2(L^2)}$ terms by the $\delta$ term on the left-hand side of  \eqref{v_testing_Blackstock}, we choose small enough $\varepsilon>0$ and small enough $R>0$ in \eqref{ball_KB} so that
\begin{equation}
	\begin{aligned}
 \|\Delta \ut\|_{L^2(L^2)}+\|\Delta u^\tau_t\|_{L^2(L^2)}
+\|\Delta u^\tau\|_{L^\infty(L^2)} \lesssim  R,
	\end{aligned}
\end{equation}
and thus the terms on the left above can be made small relative to $\delta$, independently of $\tau$. The remaining terms can be treated using Gr\"onwall's inequality, to arrive at the following result.  
\begin{theorem}[Limiting strong behavior of equations with Blackstock-type nonlinearities] \label{Thm:StrongLim_Blackstock}
	Let the assumptions of Theorem~\ref{Thm:WellP_GFE_BlckstockKuznetov} hold for \eqref{IBVP_GFE_tau_BlackstockKuznetsov} with \[(u_0^\tau, u_1^\tau)=(u_0, u_1)\] independent of $\tau$. Let $\{u^\tau\}_{\tau \in (0, \bar{\tau}]} \subset \calB^{\textup{KB}}$ be the family of solutions to \eqref{IBVP_GFE_tau_BlackstockKuznetsov} with $k_1=0$
	and let $u$ be the solution of the corresponding limiting initial boundary-value problem for the Blackstock equation:
\begin{equation} \label{Blackstock}
	\begin{aligned}
		&	\utt
			-c^2(1-2k_2 \ut) \Delta u
			-\delta   \Delta \ut+2 k_3 \nabla u \cdot \nabla \ut= f
	\end{aligned} 
\end{equation}
with $u\vert_{\partial \Omega}=0$ and
$	(u, u_t)\vert_{t=0} =(u_0, u_1).$ 	\begin{equation}
\|\nabla(u-u^\tau)\|_{L^\infty(L^2)}  +  \left\{\int_0^{T}\|\Delta (u-u^\tau)(s)\|_{L^2}^2\ds\right\}^{1/2} \leq  C \taua,
\end{equation}
where the constant $C>0$ does not depend on $\tau$.
\end{theorem}
Thus, provided the kernel and data satisfy the assumptions of Theorem~\ref{Thm:StrongLim_Blackstock}, this result establishes equation
\begin{equation}
{\tau^a}  \frakKone * (\utt^\tau-c^2\Delta u^\tau)_t+ \utt^\tau
-c^2 (1-2k_2u^\tau_t) \Delta u^\tau 
-\delta   \Delta \ut^\tau+2k_3 \nabla u^\tau \cdot \nabla u^\tau_t=f 
\end{equation}
as an approximation of the Blackstock equation (and vice versa) for small enough $\tau>0$, as well as the error one makes when exchanging them.
\section*{Acknowledgments}
The author is thankful to Barbara Kaltenbacher (University of Klagenfurt) and Mostafa Meliani (Radboud University) for many interesting discussions on fractional wave equations and for the helpful comments on an earlier version of this manuscript.
\renewcommand\appendixname{Appendix}
\begin{appendices}
	\section{Unique solvability of the semi-discrete problem} \label{Appendix:Galerkin}
We present in this appendix the proof of unique solvability of the semi-discrete problem discussed in Proposition~\ref{Prop:Wellp_Westlin}. Let $\{\phi_i\}_{i\geq 1}$ be the eigenfunctions of the Dirichlet--Laplace operator:
\begin{equation}
	- \Delta \phi_i = \lambda_i \phi_i \quad \text{in } \Omega, \qquad \phi_i=0 \quad \text{on} \ \partial \Omega,
\end{equation}
and let $V_n=\text{span}\{\phi_1, \ldots, \phi_n\}$. The approximate solution is sought in the form of
	\begin{equation}
		\begin{aligned}
			\un(t) = \sum_{i=1}^n \xin (t) \phi_i.
		\end{aligned}	
	\end{equation}
We choose approximate initial data as
	\begin{equation}
		\un_0=  \sum_{i=1}^n \xi_i^{(0, n)}(t) \phi_i,\quad \un_1= \sum_{i=1}^n \xi_i^{(1, n)}(t) \phi_i, \quad  \un_2 \sum_{i=1}^n \xi_i^{(2, n)}(t) \phi_i \in V_n,
	\end{equation}
	such that we have convergence in the following sense:
	\begin{equation} \label{convergence_approx_initial_data}
		\begin{aligned}
			&\un_0 \rightarrow u_0 \  \text{in} \ \Honetwo, \quad  \un_1 \rightarrow u_1 \ \text{in} \ \Honezero, \quad \text{and } \\
			&\un_2 \rightarrow u_2 \  \text{in} \ \Ltwo, 
		\end{aligned}
	\end{equation}
as $n \rightarrow \infty$. For each $n \in \N$, the system of Galerkin equations is then given by
	\begin{equation}
		\begin{aligned}
			\begin{multlined}[t]\taua	\sum_{i=1}^n (\frakKone * \xittt)(t) (\phi_i, \phi_j)_{L^2}+ \sum_{i=1}^n \xitt (\aaa(t)\phi_i, \phi_j)_{L^2}+ \sum_{i=1}^n \xit (\aaa_t(t)\phi_i, \phi_j)_{L^2}\\
				+c^2  \sum_{i=1}^n \xin ( \Delta \phi_i, \phi_j)_{L^2} 
				+\tau^a c^2 \sum_{i=1}^n (\frakKone * \xit)(t) (\nabla \phi_i, \nabla \phi_j)_{L^2} \\+ \delta \sum_{i=1}^n \xit(t) (\nabla \phi_i, \nabla \phi_j)_{L^2} 
				= (f(t), \phi_j)_{L^2}
			\end{multlined}	
		\end{aligned}		
	\end{equation}
	for a.e.\ $t \in (0,T)$ and all $j \in \{1, \ldots, n\}$. With $\bxi = [\xin_1 \ \ldots \  \xin_n]^T$, we can write this system in matrix form 
	\begin{equation}
		\left \{	\begin{aligned}
			& \taua M	\frakKone * \bxittt + M_\aaa \bxitt+ K \bxi+\tau^a c^2 K  \frakKone* \bxit +  (\delta K+M_{\aaa_t}) \bxit = \boldsymbol{f}, \\[1mm]
			& (\bxin, \bxit, \bxitt)\vert_{t=0} = (\boldsymbol{\xi_0}, \boldsymbol{\xi_1}, \boldsymbol{\xi_2}),
		\end{aligned} \right.
	\end{equation}
	where \[(\boldsymbol{\xi_0}, \boldsymbol{\xi_1}, \boldsymbol{\xi_2})=([\xi_1^{(0,n)} \, \ldots \, \xi_n^{(0, n)}]^T, [\xi_1^{(1,n)} \, \ldots \, \xi_n^{(1, n)}]^T, [\xi_1^{(2,n)} \, \ldots \, \xi_n^{(2, n)}]^T).\] Here $M$ and $K$ denote the standard mass and stiffness matrices, respectively, and $M_\aaa$ and $M_{\aaa_t}$ are the mass matrix with $\aaa$- and $\aaa_t$-weighted entries, respectively:
	\[
	M_{\aaa, ij}(t)=(\aaa(t) \phi_i, \phi_j)_{L^2}, \qquad M_{\aaa_t, ij}(t)=(\aaa_t(t) \phi_i, \phi_j)_{L^2}.
	\] Let the new unknown be $\bmu= \frakKone*\bxittt$. Then
	\begin{equation}
		\begin{aligned}
			&\bxitt=\, \tfrakK* \bmu + \boldsymbol{\xi_2}, \\
			&\bxit= 1*\tfrakK*\bmu+ \boldsymbol{\xi_2}t+\boldsymbol{\xi_1}, \\
			&\bxi=\,1*1*\tfrakK*\bmu+\boldsymbol{\xi_2} \frac{t^2}{2}+\boldsymbol{\xi_1}t +\boldsymbol{\xi_0}.
		\end{aligned}
	\end{equation}
	The system can then be equivalently rewritten as a system of Volterra equations:
	\begin{equation}
		\begin{aligned}
			\begin{multlined}[t] \taua	\bmu+ M^{-1} M_\aaa \tilde{\frakK}* \bmu + M^{-1}K 1*1*\tfrakK*\bmu
				+\tau^a c^2 M^{-1}K  \frakK* 1*\tfrakK*\bmu\\+   M^{-1}(\delta K+M_{\aaa_t}) 1 *\tfrakK* \bmu  =  \boldsymbol{\tilde{f}}
			\end{multlined}
		\end{aligned}
	\end{equation}
	with the right-hand side
	\begin{equation}
		\begin{aligned}
			\boldsymbol{\tilde{f}}=&\, 	\begin{multlined}[t]  M^{-1}\boldsymbol{f}-M^{-1} M_\aaa \boldsymbol{\xi_2}-M^{-1}K(\boldsymbol{\xi_2} \frac{t^2}{2}+\boldsymbol{\xi_1}t +\boldsymbol{\xi_0})\\-\tau^a c^2 M^{-1}K  \frakK*(\boldsymbol{\xi_2}t+\boldsymbol{\xi_1}) -  M^{-1}(\delta K +M_{\aaa_t})(\boldsymbol{\xi_2}t+\boldsymbol{\xi_1} )	\end{multlined}
		\end{aligned}
	\end{equation}
that belongs to $L^2(0,T)$.	By~\cite[Theorem 3.5, Ch.\ 2 ]{gripenberg1990volterra}, the system has a unique solution $\bfmu \in L^\infty(0,T)$. Next we consider
	\begin{equation}
		\left \{	\begin{aligned}
			&	\frakK*\bxittt =\bmu, \\
			& (\bfxi, \bxit, \bxitt)\vert_{t=0} =(\bxi_0, \bxi_1, \bxi_2).
		\end{aligned} \right.
	\end{equation}
	We can rewrite this problem equivalently as
	\begin{equation}
		\left \{	\begin{aligned}
			&	\bxitt = \tfrakK* \bmu +\bxi_2 \in L^\infty(0,T), \\
			& (\bfxi, \bxit)\vert_{t=0} =(\bxi_0, \bxi_1),
		\end{aligned} \right.
	\end{equation}
	which has a unique solution $\bxi \in W^{2, \infty}(0,T)$. In turn we obtain existence of a unique approximate solution $\un \in W^{2, \infty}(0,T; V_n)$. \\[2mm]
\end{appendices}
\bibliography{references}{}

\begin{thebibliography}{10}

\bibitem{alves2018moore}
{\sc M.~d.~O. Alves, A.~Caixeta, M.~A.~J. da~Silva, and J.~H. Rodrigues}, {\em
  Moore--{G}ibson--{T}hompson equation with memory in a history framework: a
  semigroup approach}, Zeitschrift f{\"u}r angewandte Mathematik und Physik, 69
  (2018), p.~106.

\bibitem{blackstock1963approximate}
{\sc D.~T. Blackstock}, {\em Approximate equations governing finite-amplitude
  sound in thermoviscous fluids}, tech. rep., General Dynamics/Electronics
  Rochester NY, 1963.

\bibitem{bongarti2020vanishing}
{\sc M.~Bongarti, S.~Charoenphon, and I.~Lasiecka}, {\em Vanishing relaxation
  time dynamics of the {J}ordan--{M}oore--{G}ibson--{T}hompson equation arising
  in nonlinear acoustics}, Journal of Evolution Equations, 21 (2021),
  pp.~3553--3584.

\bibitem{bongarti2021boundary}
{\sc M.~Bongarti, I.~Lasiecka, and J.~H. Rodrigues}, {\em Boundary
  stabilization of the linear mgt equation with partially absorbing boundary
  data and degenerate viscoelasticity}, Discrete and Continuous Dynamical
  Systems - S, 15, pp.~1355--1376.

\bibitem{bucci2019regularity}
{\sc F.~Bucci and L.~Pandolfi}, {\em On the regularity of solutions to the
  {M}oore--{G}ibson--{T}hompson equation: a perspective via wave equations with
  memory}, Journal of Evolution Equations, 20 (2020), pp.~837--867.

\bibitem{cattaneo1958forme}
{\sc C.~Cattaneo}, {\em Sur une forme de l'\'equation de la chaleur \'eliminant
  le paradoxe d’une propagation instantan\'ee}, Comptes Rendus de
  l'Acad\'emie des Sciences de Paris, 247 (1958), pp.~431--433.

\bibitem{chen2019nonexistence}
{\sc W.~Chen and A.~Palmieri}, {\em Nonexistence of global solutions for the
  semilinear {M}oore--{G}ibson--{T}hompson equation in the conservative case},
  Discrete and Continuous Dynamical Systems, 40 (2020), pp.~5513--5540.

\bibitem{compte1997generalized}
{\sc A.~Compte and R.~Metzler}, {\em The generalized {C}attaneo equation for
  the description of anomalous transport processes}, Journal of Physics A:
  Mathematical and General, 30 (1997), p.~7277.

\bibitem{crighton1979model}
{\sc D.~G. Crighton}, {\em Model equations of nonlinear acoustics}, Annual
  Review of Fluid Mechanics, 11 (1979), pp.~11--33.

\bibitem{dautray1992evolution}
{\sc R.~Dautray and J.-L. Lions}, {\em Evolution problems {I}, volume 5 of
  mathematical analysis and numerical methods for science and technology},
  1992.

\bibitem{dell2016moore}
{\sc F.~Dell'Oro, I.~Lasiecka, and V.~Pata}, {\em The
  {M}oore--{G}ibson--{T}hompson equation with memory in the critical case},
  Journal of Differential Equations, 261 (2016), pp.~4188--4222.

\bibitem{dell2017moore}
{\sc F.~Dell’Oro and V.~Pata}, {\em On the {M}oore--{G}ibson--{T}hompson
  equation and its relation to linear viscoelasticity}, Applied Mathematics \&
  Optimization, 76 (2017), pp.~641--655.

\bibitem{evans2010partial}
{\sc L.~C. Evans}, {\em Partial Differential Equations}, vol.~2, Graduate
  Studies in Mathematics, AMS, 2010.

\bibitem{fritz2018well}
{\sc M.~Fritz, V.~Nikoli{\'c}, and B.~Wohlmuth}, {\em Well-posedness and
  numerical treatment of the {B}lackstock equation in nonlinear acoustics},
  Mathematical Models and Methods in Applied Sciences, 28 (2018),
  pp.~2557--2597.

\bibitem{gripenberg1990volterra}
{\sc G.~Gripenberg, S.-O. Londen, and O.~Staffans}, {\em Volterra integral and
  functional equations}, no.~34, Cambridge University Press, 1990.

\bibitem{hamilton1998nonlinear}
{\sc M.~F. Hamilton and D.~T. Blackstock}, {\em Nonlinear acoustics}, vol.~237,
  Academic press San Diego, 1998.

\bibitem{holm2019waves}
{\sc S.~Holm}, {\em Waves with Power-Law Attenuation}, Springer, 2019.

\bibitem{jin2021fractional}
{\sc B.~Jin}, {\em Fractional differential equations}, Springer, 2021.

\bibitem{jordan2014second}
{\sc P.~M. Jordan}, {\em Second-sound phenomena in inviscid, thermally relaxing
  gases}, Discrete \& Continuous Dynamical Systems-B, 19 (2014), p.~2189.

\bibitem{kaltenbacher2015mathematics}
{\sc B.~Kaltenbacher}, {\em Mathematics of nonlinear acoustics}, Evolution
  Equations \& Control Theory, 4 (2015), p.~447.

\bibitem{kaltenbacher2021determining}
{\sc B.~Kaltenbacher, U.~Khristenko, V.~Nikoli{\'c}, M.~L. Rajendran, and
  B.~Wohlmuth}, {\em Determining kernels in linear viscoelasticity}, Journal of
  Computational Physics,  (2022), p.~111331.

\bibitem{kaltenbacher2009global}
{\sc B.~Kaltenbacher and I.~Lasiecka}, {\em Global existence and exponential
  decay rates for the {W}estervelt equation}, Discrete \& Continuous Dynamical
  Systems-S, 2 (2009), p.~503.

\bibitem{kaltenbacher2011wellposedness}
{\sc B.~Kaltenbacher, I.~Lasiecka, and R.~Marchand}, {\em Wellposedness and
  exponential decay rates for the {M}oore-{G}ibson--{T}hompson equation arising
  in high intensity ultrasound}, Control and Cybernetics, 40 (2011),
  pp.~971--988.

\bibitem{kaltenbacher2022limiting}
{\sc B.~Kaltenbacher, M.~Meliani, and V.~Nikoli{\'c}}, {\em Limiting behavior
  of quasilinear wave equations with fractional-type dissipation}, arXiv
  preprint arXiv:2206.15245,  (2022).

\bibitem{kaltenbacher2019jordan}
{\sc B.~Kaltenbacher and V.~Nikoli{\'c}}, {\em The
  {J}ordan--{M}oore--{Gi}bson--{T}hompson equation: {W}ell-posedness with
  quadratic gradient nonlinearity and singular limit for vanishing relaxation
  time}, Mathematical Models and Methods in Applied Sciences, 29 (2019),
  pp.~2523--2556.

\bibitem{kaltenbacher2020vanishing}
\leavevmode\vrule height 2pt depth -1.6pt width 23pt, {\em Vanishing relaxation
  time limit of the {J}ordan--{M}oore--{G}ibson--{T}hompson wave equation with
  {N}eumann and absorbing boundary conditions}, Pure and Applied Functional
  Analysis, 5 (2020).

\bibitem{kaltenbacher2021inviscid}
{\sc B.~Kaltenbacher and V.~Nikoli\'c}, {\em The inviscid limit of third-order
  linear and nonlinear acoustic equations}, SIAM Journal on Applied
  Mathematics, 81 (2021), pp.~1461--1482.

\bibitem{kaltenbacher2022time}
{\sc B.~Kaltenbacher and V.~Nikoli{\'c}}, {\em Time-fractional
  {M}oore--{G}ibson--{T}hompson equations}, Mathematical Models and Methods in
  Applied Sciences, 32 (2022), pp.~965--1013.

\bibitem{kaltenbacher2023vanishing}
\leavevmode\vrule height 2pt depth -1.6pt width 23pt, {\em The vanishing
  relaxation time behavior of multi-term nonlocal
  {J}ordan--{M}oore--{G}ibson--{T}hompson equations}, Nonlinear Analysis: Real
  World Applications, 76 (2024), p.~103991.

\bibitem{kennedy2005high}
{\sc J.~E. Kennedy}, {\em High-intensity focused ultrasound in the treatment of
  solid tumours}, Nature reviews cancer, 5 (2005), pp.~321--327.

\bibitem{kubica2020time}
{\sc A.~Kubica, K.~Ryszewska, and M.~Yamamoto}, {\em Time-fractional
  Differential Equations: A Theoretical Introduction}, Springer, 2020.

\bibitem{kuznetsov1971equations}
{\sc V.~P. Kuznetsov}, {\em Equations of nonlinear acoustics}, Soviet Physics:
  Acoustics, 16 (1970), pp.~467--470.

\bibitem{lasiecka2017global}
{\sc I.~Lasiecka}, {\em Global solvability of {M}oore--{G}ibson--{T}hompson
  equation with memory arising in nonlinear acoustics}, Journal of Evolution
  Equations, 17 (2017), pp.~411--441.

\bibitem{lasiecka2015moore}
{\sc I.~Lasiecka and X.~Wang}, {\em {M}oore--{G}ibson--{T}hompson equation with
  memory, part {II}: {G}eneral decay of energy}, Journal of Differential
  Equations, 259 (2015), pp.~7610--7635.

\bibitem{lasiecka2016moore}
\leavevmode\vrule height 2pt depth -1.6pt width 23pt, {\em
  {M}oore--{G}ibson--{T}hompson equation with memory, part {I}: {E}xponential
  decay of energy}, Zeitschrift f{\"u}r angewandte Mathematik und Physik, 67
  (2016), pp.~1--23.

\bibitem{liu2020new}
{\sc W.~Liu, Z.~Chen, and D.~Chen}, {\em New general decay results for a
  {M}oore--{G}ibson--{T}hompson equation with memory}, Applicable Analysis, 99
  (2020), pp.~2624--2642.

\bibitem{meliani2023}
{\sc M.~Meliani}, {\em A unified analysis framework for generalized fractional
  {M}oore--{G}ibson--{T}hompson equations: {W}ell-posedness and singular
  limits}, Fractional Calculus and Applied Analysis.
\newblock doi: 10.1007/s13540-023-00203-x.

\bibitem{meyer2011optimal}
{\sc S.~Meyer and M.~Wilke}, {\em Optimal regularity and long-time behavior of
  solutions for the {W}estervelt equation}, Applied Mathematics \&
  Optimization, 64 (2011), pp.~257--271.

\bibitem{moore1960propagation}
{\sc F.~Moore and W.~Gibson}, {\em Propagation of weak disturbances in a gas
  subject to relaxation effects}, Journal of the Aerospace Sciences, 27 (1960),
  pp.~117--127.

\bibitem{pellicer2019wellposedness}
{\sc M.~Pellicer and B.~Said-Houari}, {\em Wellposedness and decay rates for
  the {C}auchy problem of the {M}oore--{G}ibson--{T}hompson equation arising in
  high intensity ultrasound}, Applied Mathematics \& Optimization, 80 (2019),
  pp.~447--478.

\bibitem{podlubny1998fractional}
{\sc I.~Podlubny}, {\em Fractional differential equations: an introduction to
  fractional derivatives, fractional differential equations, to methods of
  their solution and some of their applications}, Elsevier, 1998.

\bibitem{prieur2011nonlinear}
{\sc F.~Prieur and S.~Holm}, {\em Nonlinear acoustic wave equations with
  fractional loss operators}, The Journal of the Acoustical Society of America,
  130 (2011), pp.~1125--1132.

\bibitem{prieur2012more}
{\sc F.~Prieur, G.~Vilenskiy, and S.~Holm}, {\em A more fundamental approach to
  the derivation of nonlinear acoustic wave equations with fractional loss
  operators (l)}, The Journal of the Acoustical Society of America, 132 (2012),
  pp.~2169--2172.

\bibitem{racke2020global}
{\sc R.~Racke and B.~Said-Houari}, {\em Global well-posedness of the {C}auchy
  problem for the 3{D} {J}ordan--{M}oore--{G}ibson--{T}hompson equation},
  Communications in Contemporary Mathematics, 23 (2021), p.~2050069.

\bibitem{said2022global}
{\sc B.~Said-Houari}, {\em Global existence for the
  {J}ordan--{M}oore--{G}ibson--{T}hompson equation in {B}esov spaces}, Journal
  of Evolution Equations, 22 (2022), pp.~1--40.

\bibitem{simon1986compact}
{\sc J.~Simon}, {\em Compact sets in the space ${L_p(0, T; B)}$}, Annali di
  Matematica pura ed applicata, 146 (1986), pp.~65--96.

\bibitem{szabo2004diagnostic}
{\sc T.~L. Szabo}, {\em Diagnostic ultrasound imaging: inside out}, Academic
  press, 2004.

\bibitem{tani2017mathematical}
{\sc A.~Tani}, {\em Mathematical analysis in nonlinear acoustics}, in AIP
  Conference Proceedings, vol.~1907, AIP Publishing LLC, 2017, p.~020003.

\bibitem{temam2012infinite}
{\sc R.~Temam}, {\em Infinite-dimensional dynamical systems in mechanics and
  physics}, vol.~68, Springer Science \& Business Media, 2012.

\bibitem{westervelt1963parametric}
{\sc P.~J. Westervelt}, {\em Parametric acoustic array}, The Journal of the
  Acoustical Society of America, 35 (1963), pp.~535--537.

\bibitem{zhang2014time}
{\sc W.~Zhang, X.~Cai, and S.~Holm}, {\em Time-fractional heat equations and
  negative absolute temperatures}, Computers \& Mathematics with Applications,
  67 (2014), pp.~164--171.

\end{thebibliography}
\bibliographystyle{siam} 
\end{document}